\documentclass[11pt]{article}
\usepackage[margin=1in]{geometry} 
\usepackage{array,multirow}
\usepackage{amssymb,amsfonts,amsmath,bbm,mathrsfs,stmaryrd,bbm,amsrefs}
\usepackage{soul,xcolor}
\usepackage{enumerate}
\usepackage{url}
\usepackage{ytableau}
\usepackage{verbatim}

\makeatletter
\newcommand*{\tline}{%
  \ifmeasuring@
  \else
    \ifodd\column@
      \expandafter\rlap
    \else
      \expandafter\llap
    \fi
    {%
      \vrule height-1ex depth \dimexpr1ex+.4pt\relax width
      \ifcase\numexpr\column@+1\expandafter\relax
      \maxcolumn@widths
      \fi
    }%
  \fi
}
\makeatother

\usepackage[colorlinks,
             linkcolor=black!75!red,
             citecolor=blue,
             pdftitle={Mackey},
             pdfauthor={Alexandru Chirvasitu},
             pdfproducer={pdfLaTeX},
             pdfpagemode=None,
             bookmarksopen=true
             bookmarksnumbered=true]{hyperref}

\usepackage{xr}

\usepackage{tikz,tikz-cd}
\usetikzlibrary{matrix,arrows,calc,decorations.pathreplacing,decorations.markings,intersections,shapes.geometric,through,fit,shapes.symbols,positioning,decorations.pathmorphing,snakes}

\renewcommand{\labelenumi}{(\arabic{enumi})}

\usepackage[amsmath,thmmarks,hyperref]{ntheorem}
\usepackage{cleveref}

\crefname{section}{Section}{Sections}
\crefformat{section}{#2Section~#1#3} 
\Crefformat{section}{#2Section~#1#3} 

\crefname{subsection}{\S}{\S\S}
\crefformat{subsection}{#2\S#1#3} 
\Crefformat{subsection}{#2\S#1#3} 

\theoremstyle{plain}

\newtheorem{lemma}{Lemma}[section]
\newtheorem{proposition}[lemma]{Proposition}
\newtheorem{corollary}[lemma]{Corollary}
\newtheorem{theorem}[lemma]{Theorem}

\theoremstyle{nonumberplain}

\newtheorem{qfbis}{\Cref{le.quot_filt} bis}
\newtheorem{qf'bis}{\Cref{le.quot_filt'} bis}

\theoremstyle{plain}
\theorembodyfont{\upshape}
\theoremsymbol{\ensuremath{\blacklozenge}}

\newtheorem{definition}[lemma]{Definition}
\newtheorem{notation}[lemma]{Notation}
\newtheorem{example}[lemma]{Example}
\newtheorem{remark}[lemma]{Remark}
\newtheorem{convention}[lemma]{Convention}

\crefname{definition}{definition}{definitions}
\crefformat{definition}{#2definition~#1#3} 
\Crefformat{definition}{#2Definition~#1#3} 

\crefname{ex}{example}{examples}
\crefformat{example}{#2example~#1#3} 
\Crefformat{example}{#2Example~#1#3} 

\crefname{remark}{remark}{remarks}
\crefformat{remark}{#2remark~#1#3} 
\Crefformat{remark}{#2Remark~#1#3} 

\crefname{convention}{convention}{conventions}
\crefformat{convention}{#2convention~#1#3} 
\Crefformat{convention}{#2Convention~#1#3}

\crefname{lemma}{lemma}{lemmas}
\crefformat{lemma}{#2lemma~#1#3} 
\Crefformat{lemma}{#2Lemma~#1#3} 

\crefname{proposition}{proposition}{propositions}
\crefformat{proposition}{#2proposition~#1#3} 
\Crefformat{proposition}{#2Proposition~#1#3} 

\crefname{corollary}{corollary}{corollaries}
\crefformat{corollary}{#2corollary~#1#3} 
\Crefformat{corollary}{#2Corollary~#1#3} 

\crefname{theorem}{theorem}{theorems}
\crefformat{theorem}{#2theorem~#1#3} 
\Crefformat{theorem}{#2Theorem~#1#3} 

\crefname{assumption}{assumption}{Assumptions}
\crefformat{assumption}{#2assumption~#1#3} 
\Crefformat{assumption}{#2Assumption~#1#3} 

\crefname{equation}{}{}
\crefformat{equation}{(#2#1#3)} 
\Crefformat{equation}{(#2#1#3)}

\theoremstyle{nonumberplain}
\theoremsymbol{\ensuremath{\blacksquare}}

\newtheorem{proof}{Proof}

\newtheorem{proof-of-uptri}{Proof of \Cref{pr.up-tri}}
\newtheorem{proof-of-simples}{Proof of \Cref{pr.simples}}
\newtheorem{proof-of-locnoe}{Proof of \Cref{pr.loc-noe}}
\newtheorem{proof-of-exts}{Proof of \Cref{th.exts}}
\newtheorem{proof-of-univ}{Proof of \Cref{th.univ}}
\newtheorem{proof-of-univ-gen}{Proof of \Cref{th.univ-gen}}

\newcommand\bC{\mathbb C}

\newcommand\bK{\mathbb K}

\newcommand\bT{\mathbb T}

\newcommand\bZ{\mathbb Z}

\newcommand\cA{\mathcal A}

\newcommand\cC{\mathcal C}
\newcommand\cD{\mathcal D}

\newcommand\cS{\mathcal S}

\newcommand\fg{\mathfrak{g}}
\newcommand\fk{\mathfrak{k}}

\newcommand\fG{\mathfrak{G}}

\newcommand\al{\aleph}

\newcommand\fgl{\mathfrak{gl}}
\newcommand\fsl{\mathfrak{sl}}
\newcommand\fo{\mathfrak{o}}
\newcommand\fsp{\mathfrak{sp}}

\DeclareMathOperator{\id}{id}
\DeclareMathOperator{\End}{\mathrm{End}}

\DeclareMathOperator{\Hom}{\mathrm{Hom}}
\DeclareMathOperator{\ext}{\mathrm{Ext}}
\DeclareMathOperator{\cor}{\mathrm{cor}}

\DeclareMathOperator{\soc}{\mathrm{soc}}
\DeclareMathOperator{\op}{\mathrm{op}}
\DeclareMathOperator{\usoc}{\underline{\mathrm{soc}}}

\DeclareMathOperator{\sym}{\cat{Sym}}

\newcommand{\cat}[1]{\textsc{#1}}

\newcommand{\qedhere}{\mbox{}\hfill\ensuremath{\blacksquare}}

\title{Representation categories of Mackey Lie algebras as universal monoidal categories}
\author{Alexandru Chirvasitu,\quad   Ivan Penkov}

\begin{document}
\setstcolor{red}

\date{}

\newcommand{\Addresses}{{
  \bigskip
  \footnotesize

  \textsc{Department of Mathematics, University at Buffalo, Buffalo,
    NY 14260-2900, USA}\par\nopagebreak \textit{E-mail address}:
  \texttt{achirvas@buffalo.edu}

  \medskip

  \textsc{Jacobs University Bremen, Campus Ring 1, 28759 Bremen, Germany}\par\nopagebreak
  \textit{E-mail address}: \texttt{i.penkov@jacobs-university.de}
}}

\maketitle

\begin{flushright}
To Yuri Ivanovich Manin on the occasion of his 80th birthday
\end{flushright}

\begin{abstract}
  Let $\bK$ be an algebraically closed field of characteristic $0$.
  We study a monoidal category $\bT_\alpha$ which is universal among
  all symmetric $\bK$-linear monoidal categories generated by two
  objects $A$ and $B$ such that $A$ has a, possibly transfinite,
  filtration.  We construct $\bT_\alpha$ as a category of
  representations of the Lie algebra $\fgl^M(V_*,V)$ consisting of
  endomorphisms of a fixed diagonalizable pairing
  $V_*\otimes V\to \bK$ of vector spaces $V_*$ and $V$ of dimension
  $\alpha$.  Here $\alpha$ is an arbitrary cardinal number.  We
  describe explicitly the simple and the injective objects of
  $\bT_\alpha$ and prove that the category $\bT_\alpha$ is Koszul.  We
  pay special attention to the case where the filtration on $A$ is
  finite.  In this case $\alpha=\aleph_t$ for $t\in\bZ_{\geq 0}$.
\end{abstract}


\vspace{.5cm}



\section*{Introduction}\label{se.intro}

In the last decade, monoidal categories of representations of infinite
matrix algebras have been studied from various points of view.  In
particular, in the paper~\cite{DPS} the category
$\mathbb{T}_{\fsl(\infty)}$ was introduced and investigated in detail.
To recall the definition of this category, fix an algebraically closed
field $\mathbb{K}$ of characteristic $0$ together with a nondegenerate
pairing (bilinear map) $\mathbf{p}:V_*\times V\to\mathbb{K}$ for some
countable-dimensional vector spaces $V_*$ and $V$ over $\mathbb{K}$.
Then $V_*\otimes V$ has an obvious structure of an associative algebra
($\big(\left(v_*\right)_1\otimes
v_1\big)\,\big(\left(v_*\right)_2\otimes
v_2\big)=\mathbf{p}\big(\left(v_*\right)_2,v_1\big)\,\left(v_*\right)_1\otimes
v_2$), and hence also of a Lie algebra.  The Lie subalgebra
$\ker\mathbf{p}\subset V_*\otimes V$ is isomorphic to the Lie algebra
$\fsl(\infty)$ (in fact, the Lie algebra $\fsl(\infty)$ can simply be
defined as $\ker\mathbf{p}$).  A quick way to define the category
$\mathbb{T}_{\fsl(\infty)}$ is to declare it the monoidal category of
all $\ker\mathbf{p}$-subquotients of finite direct sums of tensor
products of the form
$\left(V_*\right)^{\otimes n}\otimes V^{\otimes m}$.  In ~\cite{DPS}
three other equivalent definitions of this category are given: they
are all intrinsic to the Lie algebra $\fsl(\infty)$.  From the point
of view of a representation-theorist, $\mathbb{T}_{\fsl(\infty)}$ is
interesting as it is the ``limit as $q\to\infty$'' of the categories
of finite-dimensional $\fsl(q)$-modules.  Unlike the category of
finite-dimensional $\fsl(q)$-modules, $\mathbb{T}_{\fsl(\infty)}$ is
not a semisimple category.  The simple objects of
$\mathbb{T}_{\fsl(\infty)}$ are parametrized by ordered pairs
$(\mu,\nu)$ of Young diagrams, and, based on earlier work~\cite{PS} by
K. Styrkas and the second author, in~\cite{DPS} the injective objects
of $\mathbb{T}_{\fsl(\infty)}$ have been described and Koszul
self-duality of $\mathbb{T}_{\fsl(\infty)}$ has been established.

In a parallel development, the category $\mathbb{T}_{\fsl(\infty)}$
arose in the work~\cite{SS} of A. Sam and A. Snowden who took a
somewhat different point of view.  In particular, they showed that
$\mathbb{T}_{\fsl(\infty)}$ is universal among $\mathbb{K}$-linear tensor categories (see Convention \ref{cv.ccpl} below) generated by two objects $A$ and
$B$ together with a morphism $A\otimes B\to\mathbbm{1}$ into the
monoidal unit $\mathbbm{1}$ (the field $\mathbb{K}$).  It is important
to note that $\mathbb{T}_{\fsl(\infty)}$ is not a rigid tensor
category, in particular there is no nonzero morphism
$\mathbbm{1}\to V_*\otimes V$, and $\mathbb{T}_{\fsl(\infty)}$ is
universal as a general (nonrigid) $\mathbb{K}$-linear
tensor category.  This universality property of the category
$\mathbb{T}_{\fsl(\infty)}$ has been used in an essential way in the
recent study~\cite{EHS} of abelianizations of Deligne categories.  In addition, the boson-fermion correspondence has been categorified via $\bT_{\fsl(\infty)}$ in \cite{FSP}.

Motivated by representation theory, in~\cite{PS2} a larger Lie
algebra, called Mackey Lie algebra, was introduced.  Let now
$\mathbf{p}:V_*\times V\to\mathbb{K}$ be a nondegenerate pairing for
some vector spaces, not necessarily countable dimensional or having
the same dimension.  A pioneering study of such pairings was
undertaken by G. Mackey in his dissertation~\cite{M}.  The
endomorphisms of the pairing $\mathbf{p}$ form a Lie algebra which we
denote by $\fgl^M(V_*,V)$ and call the \emph{Mackey Lie algebra} of
$\mathbf{p}$: this Lie algebra is defined by formula (\ref{eq.fglM})
below, and has $V_*\otimes V$ as its ideal.

As a next step, in the work~\cite{us} we showed that, if
$\dim V_*=\dim V=\aleph_0$ (i.e., if $V_*$ and $V$ are countable
dimensional as in~\cite{DPS}), a natural category of representations
of $\fgl^M(V_*,V)$ has also a universality property.  More precisely,
in~\cite{us} we consider the category $\mathbb{T}^3_{\fgl^M(V_*,V)}$
of $\fgl^M(V_*,V)$-modules isomorphic to subquotients of finite sums
of tensor products of the form
$\left(V^*\right)^{\otimes n}\otimes V^{\otimes m}$ for $m,n\geq 0$
($V^*$ being the algebraic dual space $\Hom_\mathbb{K}(V,\mathbb{K})$).  Our
result states that $\mathbb{T}^3_{\fgl^M(V_*,V)}$ is universal among (nonrigid) $\mathbb{K}$-linear tensor categories generated
by two objects $A$ and $B$ such that $A$ has a subobject
$A_0\hookrightarrow A$ (in the case of $\mathbb{T}^3_{\fgl^M(V_*,V)}$,
we have $A=V^*=\Hom_\mathbb{K}(V,\mathbb{K})$, $B=V$, and $A_0=V_*$, the
inclusion $V_*\subset V^*$ being induced by $\mathbf{p}$).  The
structure of $\mathbb{T}^3_{\fgl^M(V_*,V)}$, both as an abelian and a
tensor category, is rather elaborate.  Its simple objects are
parametrized by triples $(\lambda,\mu,\nu)$ of Young diagrams, and
in~\cite{us} we compute the Ext-groups between simple objects and
prove Koszul self-duality for $\mathbb{T}^3_{\fgl^M(V_*,V)}$.

All of the above motivates the topic of our current study.  Our main
objective is to construct and study a
$\mathbb{K}$-linear tensor category which is universal among
abstract $\mathbb{K}$-linear tensor categories generated by two objects $A$ and $B$ such
that $A$ has an arbitrary fixed filtration.  In fact, we allow the
filtration on $A$ to be transfinite.  As it turns out, we can
construct such a universal category as a category of tensor
representations for a Mackey Lie algebra $\fgl^M(V_*,V)$, where the
dimension of both $V_*$ and $V$ equals $\alpha=\aleph_{a}$, $a$ being
the ordinal of the maximal proper subobject of $A$ in the fixed
transfinite filtration of $A$.

More precisely, we consider a nondegenerate pairing $\mathbf{p}:V_*\times V\to\bK$ where
$V_*,V$ are $\alpha$-dimensional for an arbitrary cardinal number $\alpha$, and suppose that the pairing is splitting in the sense that there are respective bases
$\big\{v^*_\kappa\big\}$, $\left\{v_{\kappa'}\right\}$ of $V_*$ and $V$ such that
$\mathbf{p}\big(v^*_\kappa,v_{\kappa'}\big)=\delta_{\kappa\kappa'}$
($\delta_{\kappa\kappa'}$ being the Kronecker delta).  The category $\bT_\alpha$ is then the minimal full monoidal subcategory of the
category of $\fg$-modules which contains $V$ and $V^*$ and is
closed with respect to subquotients.  For $\alpha=\aleph_0$, the
category $\bT_{\aleph_0}$ coincides with the category
$\bT^3_{\fgl^M(V_*,V)}$ studied in \cite{us}.   We show that the Grothendieck envelope $\bar{\bT}_{\alpha}$ of $\bT_\alpha$
is an ordered Grothendieck category according to a slightly more general definition than the one given in~\cite{us}, and use this in a crucial way
to deduce that objects of the form
\begin{equation}
	\bigotimes_{s=t}^0\,\left(V^*/V^*_{\beta^+_s}\right)_{\lambda_s}\otimes \left(V^*\right)_\mu\otimes V_\nu
	\label{eq.bigtensor}
\end{equation}
are injective in $\bT_\alpha$.  Here $\left\{\beta_t,\ldots,\beta_1,\beta_0\right\}$ is a finite (possibly empty) set of infinite cardinal
numbers such that $\beta_0<\beta_1<\ldots<\beta_t\leq \alpha$, $\beta_s^+$ stands for the successor cardinal
to $\beta_s$, and $\lambda_t,\ldots,\lambda_1,\lambda_0,\mu,\nu$ are
Young diagrams; $\bullet_\lambda$ denotes the Schur
functor associated with a Young diagram $\lambda$.  We show that the objects (\ref{eq.bigtensor}) have simple socles, and that the so obtained  simple modules exhaust (up to isomorphism) the simple objects of $\bT_\alpha$.  

In the case when $\alpha=\aleph_t$ for some nonnegative integer $t$ , we present an explicit combinatorial formula for the multiplicity of a simple module in an injective hull of another simple module.  This generalizes the corresponding multiplicity formulas from~\cite{PS} and~\cite{us}.  Our next result is that, for a general $\alpha$, the
category $\bT_\alpha$ is Koszul in the sense that its Ext-algebra
\begin{equation*}
	\bigoplus_{T',T\text{ simple},\,p\geq 0}\,\ext^p\left(T',T\right)  
\end{equation*}
is generated in degree one.  In the last section we use the Koszulity of $\mathbb{T}_\alpha$ to prove that $\mathbb{T}_\alpha$ possesses the universality property stated above.  

Finally, we should mention that in Section~\ref{se.ex} we present another application of the more general notion of ordered tensor category introduced in this paper: we point out that the category $\widetilde{\mathrm{Tens}}_{\fg}$ introduced and studied in~\cite{PS1} falls under the new definition, and we prove that its injective objects are nothing but arbitrary direct sums of the indecomposable injectives described in~\cite{PS1}.

\subsection*{Acknowledgements}

The first author was partially supported by NSF grant DMS-1565226.
The second author acknowledges support from DFG through grant PE
980/6-1.

\section{Background}\label{se.prel}

Let $\bK$ be a field.  Except in Section \ref{se.abs}, $\bK$
is assumed to be algebraically closed and of characteristic $0$.  All
vector spaces (including Lie algebras) are assumed to be defined over
$\bK$.  If $U$ is a vector space, we set $U^*=\Hom_\bK(U,\bK)$ and $\End(U)=\End_\bK(U)$.  We also abbreviate $\otimes_\bK$ to $\otimes$.  All additive categories considered are understood to be linear over $\bK$, and all additive functors are assumed to preserve
this structure.   

One way to define the
finitary Lie algebra $\fsl(\infty)$ is as the inductive limit
of a chain of embeddings
	\begin{equation}
		\fg_1\hookrightarrow\fg_2\hookrightarrow\ldots
		\label{eq.inclusions}
	\end{equation}
	where $\fg_q=\fsl(q+1)$.  Similarly, $\fo(\infty)$ and
        $\fsp(\infty)$ can be defined as the inductive limits of
        respective chains (\ref{eq.inclusions}) where $\fg_q=\fo(q)$ or
        $\fg_q=\fsp(2q)$.

	A natural representation of $\fsl(\infty)$, $\fo(\infty)$, and
        $\fsp(\infty)$ is a direct limit of natural representations
        $V_q$ of $\fg_q$.  For $\fg=\fo(\infty)$ (respectively, for
        $\fsp(\infty)$), up to isomorphism, there is a unique natural
        representation $V=\lim\limits_{\longrightarrow} V_q$
        (respectively, $V=\lim\limits_{\longrightarrow} V_{2q}$),
        while for $\fg=\fsl(\infty)$ there are two nonisomorphic
        natural representations
        $V=\lim\limits_{\longrightarrow} V_{q+1}$ and
        $V_*=\lim\limits_{\longrightarrow} V_{q+1}^*$; here, $V_{q+1}$
        denotes the space of column vectors of length $q+1$,
        considered as an $\mathfrak{sl}(q+1)$-module.
%
	
	Let now $V_*$ and $V$ be abstract vector spaces and $\mathbf{p}:V_*\times V\to\bK$
        be a nondegenerate bilinear map, or simply pairing, of the
        vector spaces $V_*$ and $V$.  Note that $\mathbf{p}$ induces injective
        linear operators $V_*\hookrightarrow V^*$ and
        $V\hookrightarrow \left(V_*\right)^*$.  The Mackey Lie algebra $\fgl^M(V_*,V)$
        is by definition the Lie algebra of endomorphisms of the
        pairing $\mathbf{p}$, i.e.,
	\begin{equation}
	\fgl^M(V_*,V)=\big\{x\in\End\left( V_*\right)\,|\,x^*(V)\subset V\big\} = \big\{ y\in \End(V)\,|\,y^*(V_*)\subset V_*\big\}\,,
		\label{eq.fglM}
	\end{equation}
	where here $^*$ indicates dual linear operator.  Note that $V_*\otimes V$ is an ideal in $\fgl^M\left(V_*,V\right)$.

	If $V_*$ and $V$ are both countable dimensional, then it is a
        result of G. Mackey \cite{M} that a nondegenerate pairing
        $\mathbf{p}$ is unique up to isomorphism.  In this case, the
        Mackey Lie algebra $\fgl^M(V_*,V)$ is isomorphic to the Lie
        algebra of infinite matrices with finite rows and columns.
        See \cite{PS2} for results concerning $\fgl^M(V_*,V)$ and its
        representations.
	
	The above result of Mackey provides also an alternative
        definition of the Lie algebra $\mathfrak{sl}(\infty)$.  This
        definition was already mentioned in the Introduction.  If
        $\dim V_*=\dim V= \aleph_0$, the pairing $\mathbf{p}$ is
        unique up to isomorphism of pairings, and we can set
        $\mathfrak{sl}(\infty)=\ker\mathbf{p}$.  Then $V$ and $V_*$
        are the two nonisomorphic natural representations of
        $\mathfrak{sl}(\infty)$ which were introduced above as direct
        limits.  In the rest of the paper we will have both of these
        interpretations of $V$ and $V_*$ in mind.  Note also that $V$
        and $V_*$ admit a pair of dual bases, i.e. bases
        $\left\{v_q\right\}\subset V$,
        $\left\{v_q^*\right\}\subset V_*$ for $q\in\bZ_{>0}$, such
        that $\mathbf{p}\left(v^*_{q'},v_q\right)=\delta_{q'q}$.  In
        what follows we will think of $V_*$ as a subspace of $V^*$ (the embedding $V_*\subset V^*$ being induced by the pairing
        $\mathbf{p}$).  If $V$ is equipped with a nondegenerate
        symmetric (or antisymmetric bilinear form, then
        $\mathfrak{o}(\infty)$ is defined as the Lie algebra
        $\Lambda^2(V)$ and $\mathfrak{sp}(\infty)$ is defined as the
        Lie algebra $S^2(V)$.  The spaces $\Lambda^2(V)$ and $S^2(V)$
        are endowed with respective Lie algebra structures via the
        corresponding forms.
	
	Let $\fg$ be any Lie algebra, $M$ be any $\fg$-module, and
        $\fk\subset \fg$ be a Lie subalgebra.  Recall that $\fk$
        \emph{acts densely on} $M$ if for any finite set of vectors
        $m_1,\ldots,m_q\in M$ and any $g\in\fg$, there is $k\in\fk$
        such that $g\cdot m_s=k\cdot m_s$ for $s=1,\ldots,q$.  Below
        we use the fact that if $\fk$ acts densely on $M$ then $\fk$
        acts densely on any $\fg$-subquotient of the tensor algebra
        $\mathrm{T}(M)$ \cite[Lemma 7.3]{PS2}.

	We recall also that there is a well-defined Schur functor $\bullet_\lambda$ for any Young diagram, or partition, $\lambda=\left(\lambda_1\geq \lambda_2\geq \ldots\geq \lambda_q\right)$,
	$$\bullet_\lambda:\text{Vect}\to\text{Vect}\,,$$
	$\text{Vect}$ being the category of vector spaces over $\bK$.
        By definition, $V_\lambda$ is a direct summand of the tensor
        power $V^{\otimes|\lambda|}$, where
        $|\lambda|=\lambda_1+\ldots+\lambda_q$.  For the precise
        definition see \cite{Fulton}.  If $V$ is a $\fg$-module for
        some Lie algebra $\fg$, then $V_\lambda$ has a natural
        structure of a $\fg$-module (as a $\fg$-submodule of
        $V^{\otimes |\lambda|}$).

        If $\mathcal{C}$ is an abelian category, a \emph{chain} of
        objects of $\mathcal{C}$ is a set of objects
        $\left\{A_\sigma\right\}$ such that, for any pair of objects
        $A_{\sigma_1}$ and $A_{\sigma_2}$, precisely one noninvertible
        monomorphism $A_{\sigma_1}\to A_{\sigma_2}$ or
        $A_{\sigma_2}\to A_{\sigma_1}$ is fixed.  This endows the set
        of indices $\{\sigma\}$ with a linear order:
        $\sigma_1 <\sigma_2$ if a noninvertible monomorphism
        $A_{\sigma_1}\to A_{\sigma_2}$ is fixed.  Given an object $A$
        of $\mathcal{C}$, we say that $A$ is endowed with a
        \emph{transfinite filtration} if a chain of subobjects
        $\left\{A_\sigma\right\}$ of $A$ is given such that the linear
        order on the set of indices $\{\sigma\}$ is a well-order.

        For background on Grothendieck categories we refer the reader
        to \cite[$\S$1.1]{grk} or \cite[$\S$2.8]{groth}. In any such
        category every object $X$ has a \emph{transfinite socle
          filtration}: the socle of $X$, $\soc X$, is the maximal
        semisimple subobject of $X$ (the sum of all semisimple
        subobjects of $X$), and then the socle filtration is built by
        transfinite induction each time taking the pullback in $X$ of
        the socle of the relevant quotient.  We have
$$0 \subset \soc X=\soc^1 X \subset \soc^2 X =\pi_1^{-1}\big(X/(\soc X)\big)\subset\ldots\subset \soc^{\aleph_0}X=\pi^{-1}_{\aleph_0}\left(\lim\limits_{\longrightarrow}\,\left(\soc^qX\right)\right)\subset\ldots$$
where $\pi_1:X\to X/(\soc X)$ and
$\pi_{\aleph_0}:X\to
X/\left(\lim\limits_{\substack{{\longrightarrow}\\{q<\aleph_0}}}\,\left(\soc^qX\right)\right)$
are the canonical projections.  For $q \in\mathbb{Z}_{\geq 1}$, we denote by $\usoc^qX$ the $q$-th layer $\soc^q X/\soc^{q-1} X$ of the transfinite socle filtration of $X$.

\section{Ordered Grothendieck categories}\label{se.abs}

Here we extend the notion of ordered Grothendieck category introduced
in \cite{us} to the infinite-length setting.

First, recall that a family of objects $\left\{Z_\kappa\right\}$ of an
abelian category $\mathcal{C}$ is a \emph{family of generators} if for
any two distinct morphisms $A\overset{\varphi}{\longrightarrow}B$ and
$A\overset{\psi}{\longrightarrow}B$ in the category $\mathcal{C}$,
there is an object $Z_\kappa$ together with a morphism
$\gamma:Z_\kappa\to A$ so that the compositions $\varphi\circ\gamma$
and $\psi\circ\gamma$ are distinct, see \cite[$\S$V.7]{cats} or
\cite[discussion preceding Proposition 1.2.2]{groth}. Next, recall the
following notion (see \cite[$\S$5.6]{groth} or \cite[$\S$41]{BW} for a
discussion in the context of module categories).

\begin{definition}\label{def.sa}
  An object $X$ in a Grothendieck category if {\it semi-artinian} if
  every nonzero quotient of $X$ has nonzero socle.  A Grothendieck
  category is {\it semi-artinian} if it has a set of semi-artinian
  generators.
\end{definition}

Let $\cC$ be Grothendieck category and $(I,{{\preceq}})$ a poset. Let
also $X_i\in \cC$, $i\in I$ be a collection of semi-artinian objects,
and $\cS_i$ be the set of isomorphism classes of simple subobjects of
$X_i$.

Throughout, we assume that the opposite poset $(I,{{\preceq}})^{\op}$
(i.e. {{the set}} $I$ equipped with the partial order opposite to
${{\preceq}}$) is {\it well ordered}.  This means that every nonempty
totally ordered subset of $I$ has a largest element with respect to
${{\preceq}}$.

{The} following definition generalizes Definition 2.1 in~\cite{us}.
\begin{definition}\label{def.ordered}
  The above structure makes $\cC$ an {\it ordered Grothendieck
    category} provided the following conditions hold{{:}}
  \begin{enumerate}[(i)]
  \item every object in $\cC$ is isomorphic to a subquotient of a
    {direct} sum
    $\bigoplus\limits_{ i \in I'\subset I}X_i^{\oplus \gamma_i}$ {for
      some subset $I'\subset I$ and some cardinal numbers $\gamma_i$};
  \item the sets $\cS_i$ are disjoint, and they exhaust the
    isomorphism classes of simple objects of $\cC$;
    \item simple subquotients of $X_i$ which are not subobjects of $\soc X_i$ belong to $\cS_j$ for $j\prec i$;
    \item each $X_i$ decomposes as a direct sum of subobjects with simple socle;
    \item for all $i{{\succ}}j$ the maximal subobject
      $Y_{i{{\succ}}j}\subseteq X_i$ whose simple constituents are in
      various $\cS_k$ {for} $i{{\succeq}} k{{\succ}}j$, is the joint
      kernel of a family of morphisms $X_i\to X_j$.
  \end{enumerate}
\end{definition}

If $U$ is a semisimple subobject in $\soc X_i$, then by $\widetilde{U}$ we
denote the direct summand of $X_i$ such that $\soc \widetilde{U}=U$;
the existence of $\widetilde{U}$ is guaranteed by condition (iv).

%

\begin{remark}\label{re.is-sa}
  An ordered Grothendieck category is semi-artinian.  This follows
  from the observation that condition (i) of \Cref{def.ordered}
  together with our assumption that $X_i$ are semi-artinian ensure
  that $\cC$ is semi-artinian.
  The proof can be found in \cite[41.10 (4)]{BW} (or in the original
  source \cite[27.5, 32.5]{W} cited therein) for module categories;
  the general result is analogous.
\end{remark}

\begin{definition}
  If the $X_i$ in Definition~\ref{def.ordered} are of finite length
  (or, equivalently, if $X_i$ satisfy Definition 2.1 in~\cite{us}), we
  say that $\mathcal{C}$ is a \emph{finite ordered Grothendieck
    category}.
\label{def.finordgroth}
\end{definition}

The next proposition shows that checking the properties (i)--(v) on a
set of objects $\left\{X_i\right\}$ suffices to describe, up to
isomorphism, all injective hulls of simple objects in $\mathcal{C}$ as
direct summands of $X_i$ (cf. \cite[Proposition 2.3]{us}).

\begin{proposition}\label{pr.inj-hull}
  In the setup of \Cref{def.ordered}, for any $i\in I$ and any simple
  object $U\in \cS_i$ the object $\widetilde{U}$ is an injective hull
  of $U$.
\end{proposition}
\begin{proof}
  Let $U\subset J$ be an essential extension such that $J$ is a
  subquotient of a direct sum
  $X=\bigoplus\limits_{j\in I'\subset I} X_j^{\oplus {\gamma}_j}$ for
  some cardinal numbers ${\gamma}_j$. It suffices to show that $J$
  admits a monomorphism into $\widetilde{U}$.

  Writing $J$ as a subobject of a suitable quotient $Z$ of $X$, we can
  factor out a direct complement of $U$ in $\soc Z$ thus reducing to
  the case where $Z$ itself has socle $U$. Therefore, upon
  substituting $Z$ for $J$, we can assume that $J$ is a quotient of
  $X$.

  Now consider the largest subobject $K$ of $X$ whose simple
  constituents lie in various $\cS_k$ for $k{\succ} i$. First, since
  no subquotient of $K$ is in $\cS_i$, $K$ is automatically a
  subobject of the kernel of the epimorphism $X\to J$. Additionally,
  condition (v) in \Cref{def.ordered} ensures the existence of a
  morphism from $X$ into a product of $X_i^\gamma$ with kernel equal
  to $K$.

  Condition (iii) of \Cref{def.ordered} and the fact that the socle of
  $J$ is $U\in \cS_i$ now imply that a subobject of a single factor
  $X_i$ of $X_i^\gamma$ admits an epimorphism to $J$. Finally, this
  implies via the decomposition in condition (iv) of
  \Cref{def.ordered} that a subobject of $\widetilde{U}\subset X_i$
  admits an isomorphism with $J$.
\end{proof}

\begin{corollary}\label{cor.indec-inj}
  Under the conditions of \Cref{pr.inj-hull}, the indecomposable injective objects
  of $\cC$ are, up to isomorphism, precisely the indecomposable summands of the various
  objects $X_i$. 
\end{corollary}
\begin{proof}
  The indecomposable injectives are injective hulls of the simple
  objects, and these are precisely the objects in $\cS_i$, $i\in
  I$. The conclusion now follows from \Cref{pr.inj-hull}. 
\end{proof}

One characteristic of ordered Grothendieck categories that will be
important for us below is a certain ``upper triangular'' character of
Ext-groups between simple objects, which limits the possibilities for such Ext-groups to be nontrivial. 

In order to state the result, we need the following definition (cf.
e.g. \cite[$\S$2.2]{us}).

\begin{definition}\label{def.def}
  Let $i{\preceq} j$ be two elements in a poset $(I,{\preceq})$. The {\it defect}
  $d(i,j)$ is the supremum in $\mathbb{Z}_{\geq0}\cup\{\infty\}$ of the set of
  nonnegative integers $q$ for which there is a chain
  \begin{equation*}
    i=i_0{\prec i_1\prec \cdots\prec} i_q=j
  \end{equation*}
in $I$.
\end{definition}

\begin{remark}\label{re.def}
  Note that for $i{\preceq j\preceq} k$ we have $d(i,j)\ge d(i,k)+d(k,j)$. 
\end{remark}

We can now state the result alluded to above.

\begin{proposition}\label{pr.up-tri}
  Let $T\in \cS_j$ and $T'\in \cS_i$ be two simple objects and suppose
  $\ext^p(T',T)\ne 0$ for some $p\ge 0$. Then $i{\preceq} j$ and $d(i,j)\ge p$.   
\end{proposition}

Before delving into the proof, it will be convenient to introduce a notation and a definition. 

\begin{notation}\label{not.inj}
  For an arbitrary object $U\in \cC$ we denote by $\widetilde{U}$ {an
  injective hull of $U$ in} $\cC$.  By Proposition \ref{pr.inj-hull}, this is consistent with the previous usage of the notation $\widetilde{U}$.
\end{notation}

\begin{definition}\label{def.corona}
  For an object $U\in \cC$ its {\it corona} $\cor U$ in $\cC$ is
  $\widetilde{U}/U$. 
\end{definition}

\begin{proof-of-uptri}
The case $p=0$ is not interesting, so suppose $p\ge 1$. If $p=1$, then
the desired conclusion (namely, $i\prec j$) follows from condition (iii) of
\Cref{def.ordered}. 

Now suppose $p>1$. Then the long exact sequence of Ext-groups corresponding to the exact sequence
$$0 \to T\to \tilde{T}\to \cor T \to 0$$
yields an identification
\begin{equation}\label{eq:p_1}
  \ext^p(T',T)\cong \ext^{p-1}(T',\cor T). 
\end{equation}
By condition (iii) of \Cref{def.ordered} again, all simple
subquotients of
$\cor T$ belong to various $\cS_{\ell}$ for ${\ell \prec} j$.

If $p-1=1$, then some of these indices ${\ell}$ will be strictly
larger than $i$ (because of the nonvanishing of the right-hand side
$\ext^1$ of \Cref{eq:p_1}).  Therefore
\begin{equation*}
  i {\prec \ell \prec} j\,,
\end{equation*}
in particular, $i\prec j$ and $d(i,j)\geq 2$.
Otherwise, we can continue the process, passing to the double corona
$\cor(\cor T )$, until $p$ has been whittled down to $1$.
\end{proof-of-uptri}

The reason why we referred to the result of \Cref{pr.up-tri} as `upper
triangularity' is the following fragment of the statement, which we
isolate for emphasis; it says that nonvanishing Ext-functors are
unidirectional with respect to the poset $(I,{\preceq})$.

\begin{corollary}\label{cor.up-tri}
  If $i{\succ}j\in I$ and $T\in \cS_j$, $T'\in \cS_i$, then
\begin{equation*}
  \ext^p(T',T)=0\text{ for } p>0\,.
\end{equation*} \qedhere
\end{corollary}


Note that, although all indecomposable injective objects of $\mathcal{C}$ are described by Corollary~\ref{cor.indec-inj}, arbitrary injective objects need in general not be sums of
indecomposable injectives. We now identify some
sufficient conditions that ensure that $\cC$ is better behaved in this
sense; we then apply these results to specific ordered Grothendieck categories.

\begin{proposition}\label{pr.loc-noe}
  If each $X_i$ is a union of its finite-length subobjects then, up to isomorphism,  the
  injective objects in $\cC$ are precisely arbitrary direct sums of
  indecomposable direct summands of the $X_i$. 
\end{proposition}

Before going into the proof, recall the following notions (see
e.g. \cite[$\S$5.7 and $\S$5.8]{groth}).


\begin{definition}\label{def.loc-noe}
  Let $X$ be an object of a Grothendieck category $\cC$.
  \begin{enumerate}
   \item $X$ is {\it noetherian} if it satisfies the ascending chain condition on subobjects;
   \item $\cC$ is {\it locally noetherian} if it has a set of
      noetherian generators.
  \end{enumerate}
\end{definition}


The reason why \Cref{def.loc-noe} is relevant to \Cref{pr.loc-noe} is
that it precisely captures the conditions that give us the kind of
control over arbitrary injective objects alluded to above, as the
following result shows (this is an abbreviated version of
\cite[Theorems 5.8.7, 5.8.11]{groth}).

\begin{proposition}\label{pr.loc-noe-aux}
For a Grothendieck category $\cC$ the following conditions are equivalent{:}
\begin{enumerate}
  \item $\cC$ is locally noetherian;
  \item {the injective objects in $\cC$} are precisely the {arbitrary}
    direct sums of indecomposable injective objects.
\end{enumerate}
\qedhere
\end{proposition}

\begin{proof-of-locnoe}
The hypothesis ensures that the finite-length subquotients of the
$X_i$ form a set of generators, and hence $\cC$ is locally noetherian. Our claim follows now from \Cref{pr.loc-noe-aux}. 
\end{proof-of-locnoe}

We end this section with a discussion of how the present material
relates to the notion of a {\it highest weight category} in the sense of \cite[Definition 3.1]{CPS}. First, note that \Cref{def.ordered}
specializes to \cite[Definition 2.1]{us} when the poset $I$ has the
property that down-sets
\begin{equation*}
  I_{{\preceq} i}:=\{j\in I\ |\ j{\preceq} i\}
\end{equation*}
are finite. Moreover, \cite[Proposition 2.16]{us} shows that in that
case an ordered Grothendieck category is a highest weight category. That result extends virtually
verbatim to the present setting provided $I$ is {\it
  interval-finite}, i.e. all intervals
\begin{equation*}
  [i,k] = \{j\in I\ |\ i{\preceq j\preceq} k\},\ i,k\in I
\end{equation*}
are finite. We record the resulting statement here.

\begin{proposition}\label{pr.hw}
  An ordered Grothendieck category based on an interval-finite poset
  $(I,{\preceq})$ is a highest weight category. \qedhere
\end{proposition}

\section{A first application}\label{se.ex}

Before moving on to our main object of study, we {would like to point out that the material in
  \Cref{se.abs} applies} to an interesting category studied in
\cite{PS1}. Specifically, recall that for
\begin{equation*}
  \fg=\fsl(\infty), \fo(\infty), \fsp(\infty)
\end{equation*}
the category $\widetilde{\mathrm{Tens}}_\fg$ is defined as
the full subcategory of $\fg\text{-mod}$ consisting of integrable
modules $M$ of finite Loewy length such that the algebraic dual
$M^*$ is also {an integrable
  $\fg$-module} of finite Loewy length. It can be shown that
$\widetilde{\mathrm{Tens}}_\fg$ is closed under dualization, taking
subobjects, quotient objects, and extensions (and hence also finite
direct sums).

\begin{definition}
  For $\fg$ as above, we denote by $\cC_\fg$ the smallest full, exact Grothendieck subcategory of $\fg\text{-mod}$ containing $\widetilde{\mathrm{Tens}}_\fg$. 
\end{definition}

  The category $\cC_\fg$ is simply the full subcategory of $\fg\text{-mod}$ whose objects are {sums} of objects in $\widetilde{\mathrm{Tens}}_\fg$. 

\begin{remark}\label{re.sigma}
For an algebra $A$, the smallest Grothendieck category of modules containing a given $A$-module $M$ (constructed essentially as we have just described) is sometimes denoted by $\sigma(M)$ in the literature{,} e.g. \cite[$\S$41]{BW}.  We can think of $\widetilde{\mathrm{Tens}}_\fg$ as $\sigma(M)$ where $A$ is the enveloping algebra of $\fg$ and $M$ is the direct sum of a set of generators for $\cC_\fg$.
\end{remark}

We now explain how $\cC_\fg$ fits into the framework of \Cref{se.abs}.
Recall~\cite{DPS} that $\bT_\fg$ is the full subcategory of $\fg\text{-mod}$ which consists of $\fg$-modules isomorphic to finite-length subquotients of finite direct sums of the form $\mathrm{T}\left(V\oplus V_*\right)^{\oplus q}$ for $q\in\mathbb{Z}_{>0}$ (for $\mathfrak{g}=\fo(\infty),\fsp(\infty)$, one can replace $\mathrm{T}\left(V\oplus V_*\right)^{\oplus q}$ simply by $\mathrm{T}(V)^{\oplus q}$).  
The category $\bT_\fg$ is equipped with an auto-equivalence
$M\mapsto M_*$ (which is the identity in the case of $\fg=\fo$ or
$\fsp$) induced by the automorphism of $\fg$ arising from switching
$V$ and $V_*$. Let $F$ denote the composition of functors
\begin{equation*}
  M\mapsto {M_*\mapsto} (M_*)^*\,.
\end{equation*}

The poset $(I,{\preceq})$, relevant for the category $\mathcal{C}_\fg$, consists of all pairs $(n,m)$ of nonnegative
integers, where
\begin{equation*}
  (n,m) {\preceq} (n',m')\iff n\le n' \text{ and } m\le m'.
\end{equation*}
{For} $(n,m)\in I$ we define 
\begin{equation*}
  X_{n,m} := F\big((V_*)^{\otimes n}\otimes V^{\otimes m}\big). 
\end{equation*}

%
%
%

Then the contents of
\cite[$\S$6]{PS1} amount to the fact that $\cC_\fg$ satisfies the
conditions of \Cref{def.ordered}; we leave the easy verification to the
reader. Condition (v), for instance, which is perhaps the least
obvious, follows from \cite[Lemma 6.6]{PS1} by taking the family of morphisms
required by condition (v) to be the family of all morphisms $X_i \to X_j$ where $i=(n',m') \succ j=(n,m)$.

%
%
%

The poset $(I,\preceq)$ is interval-finite, hence $\mathcal{C}_\fg$ is a highest-weight category according to Proposition~\ref{pr.hw}.  

Next, note that the finite-length subobjects of finite direct sums of objects $X_{n,m}$ form a set of
generators for {$\widetilde{\mathrm{Tens}}_\fg$}, and hence the same is true of $\cC_\fg$ by the
definition of the latter. This shows that $\cC_\fg$ is locally
noetherian in the sense of \Cref{def.loc-noe}, and {therefore}
according to \Cref{pr.loc-noe-aux} we have

\begin{proposition}
  The injective objects in $\cC_\fg$ are precisely {arbitrary} direct sums
  of indecomposable {injectives.}  \qedhere
\label{prop_inj}
\end{proposition}

{In \cite[Corollary 6.7]{PS1}, the indecomposable injectives in the
  category $\widetilde{\mathrm{Tens}}_\fg$ have been described explicitly as being
  isomorphic to direct summands of $X_{n,m}$, and hence
  Proposition~\ref{prop_inj} classifies the injective objects in
  $\cC_\fg$.}  Moreover, Proposition~\ref{prop_inj} is an essential
improvement of Theorem 6.15 in~\cite{PS1} which claims the existence of a
certain finite filtration on any injective object in
$\widetilde{\mathrm{Tens}}_\fg$.


\section{The categories $\bT_\alpha$ and $\bar{\bT}_\alpha$}\label{se.ta}

We now introduce a series of ordered Grothendieck categories which we
study throughout the rest of {the} paper. The general setting is as
follows: $\alpha$ is an arbitrary infinite cardinal number, $V$ and $V_*$ are
$\alpha$-dimensional complex vector spaces, and
\begin{equation*}
  \mathbf{p} :V_*\otimes V\to \bK
\end{equation*}
is a nondegenerate pairing diagonalizable in the sense that there are bases
$\left\{v^*_\kappa\right\}$ of $V_*$ and $\left\{v_{\kappa'}\right\}$ of $V$ such that $\mathbf{p}\left(v_\kappa^*,v_{\kappa'}\right)=\delta_{\kappa\kappa'}$.

Fixing the bases $\left\{v^*_\kappa\right\}$ and $\left\{v_\kappa\right\}$, and an arbitrary total order on the set $\Sigma$ of indices $\kappa$, allows us to think of the elements of $V$ as size-$\alpha$ column vectors with
finitely many nonzero entries, and of the {elements} of $V^*$
as arbitrary size-$\alpha$ row vectors; of those, the elements of
$V_*$ are precisely the row vectors with finitely many nonzero
entries.  

{By $\alpha^+$ we denote the successor cardinal
  to $\alpha$.}  For each infinite cardinal $\beta\le \alpha^+$ we denote by
$V^*_\beta\subset V^*$ the subspace consisting of row vectors with
strictly fewer than $\beta$ nonzero entries.   In this way we have a transfinite filtration
\begin{equation}
0\subset V^*_{\aleph_0}\subset \ldots \subset V^*_\alpha \subset V^*\,.
\label{eq.transfilt}\end{equation}
Note that $V^*_{\alpha^+}=V^*$ and $V^*_{\aleph_0}=V_*$ by definition. 

Let $\fgl^M$ be the Mackey Lie algebra of the pairing $\mathbf{p}$.
Using the bases $\left\{v^*_\kappa\right\}$ and
$\left\{v_\kappa\right\}$, and the total order on $\Sigma$, we can
identify $\fgl^M(V_*,V)$ with $\alpha\times\alpha$-matrices with
finite rows and columns.  Every infinite cardinal $\beta\leq\alpha^+$
yields an ideal $\fgl^M_\beta$ in $\fgl^M$: it consists of matrices in
$\fgl^M$ with at most $\beta$ nonzero rows and columns, or
equivalently with strictly fewer than $\beta$ nonzero entries.  The
action of $\fgl^M$ on $V$ is nothing but multiplication of matrices.
On $V^*$ the action of $\fgl^M$ is given by the formula
$$g\cdot v=-vg\text{ for }g\in\fgl^M, v\in V^*\,.$$
Clearly $\fgl^M\cdot V_\beta^*\subseteq V_\beta^*$, i.e., the filtration (\ref{eq.transfilt}) of $V^*$ is $\fgl^M$-stable.  

Recall that, for a Young diagram $\mu$ and an object
$M \in \bT_{\alpha}$, we denote by $M_\mu$ the image of $M$ through
the Schur functor associated to $\mu$. Moreover, given Young diagrams
$\mu$ and $\nu$, we denote by $V_{\mu,\nu}$ the space of {\it traceless
  tensors} in $(V_*)_\mu\otimes V_\nu$, i.e. those annihilated by all
compositions
\begin{equation*}
  (V_*)_\mu\otimes V_\nu \subseteq (V_*)^{\otimes |\mu|}\otimes V^{\otimes |\nu|} \to (V_*)^{\otimes (|\mu|-1)}\otimes V^{\otimes (|\nu|-1)}
\end{equation*}
where the right-hand arrow ranges over the $|\mu|\cdot|\nu|$
possible applications of $\mathbf{p}$.

\begin{definition}\label{def.t}
$\bT_\alpha$ is the smallest
full monoidal {subcategory (with respect to $\otimes$) of ${\fg\text{-mod}}$} which contains $V$ and
$V^*$, and is closed under taking subquotients.
\label{def.bTcat}
\end{definition}

We will also work with the Grothendieck envelope
$\overline{\bT}_\alpha$ obtained as the full subcategory of
${\fgl^M\text{-mod}}$ {with objects arbitrary sums} of objects in
$\bT_\alpha$ (see \Cref{re.sigma}). We embark below on a study of
$\bT_\alpha$ and $\overline{\bT}_\alpha$, our first goal being to
show that the latter category fits into the framework of
\Cref{se.abs}.

\subsection{Simple objects}\label{subse.simple}

Our aim here is to prove the following classification of the simple
objects in the category $\bT_\alpha$.

\begin{proposition}\label{pr.simples}
  Let $t$ be a nonnegative integer such that there exist infinite
  cardinal numbers $\beta_t,\ldots,\beta_0$ with
  $\beta_0<\cdots<\beta_t\leq \alpha$. Then, given Young diagrams
  \begin{equation*}
    \lambda_t,\cdots,\lambda_0, \mu,\nu,
  \end{equation*}
the object\footnote{Since in the expression $V_{(\beta_t,\lambda_t),\cdots,(\beta_0,\lambda_0),\mu,\nu}=
   \bigotimes_{s=t}^0 (V^*_{\beta_s^+}/V^*_{\beta_s})_{\lambda_s}\otimes V_{\mu,\nu}$ the indices of the cardinal numbers $\beta_s$ decrease from left to right, in what follows we will often see tensor product or summation formulas with indices ranging from $t>0$ to $0$.}
\begin{equation}\label{eq:tensor}
   { V_{(\beta_t,\lambda_t),\cdots,(\beta_0,\lambda_0),\mu,\nu}:=}
   \bigotimes_{s=t}^0 (V^*_{\beta_s^+}/V^*_{\beta_s})_{\lambda_s}\otimes V_{\mu,\nu}
\end{equation}
is simple over $\fgl^M${, and its endomorphism algebra in $\bT_\alpha$
  is $\bK$.  Moreover,} the objects obtained for distinct choices of
cardinals or Young diagrams are mutually nonisomorphic.
\end{proposition}

We work in stages towards a proof. First, we have

\begin{lemma}
  The object $V_{\mu,\nu}$ has no nonzero proper subobjects and its
  endomorphism algebra over $\fgl^M$ is $\bK$.  
\end{lemma}
\begin{proof}
  As a consequence of \cite[Theorem 2.2]{PS} and \cite[Theorem
  5.5]{PS2}, $V_{\mu,\nu}$ is simple over the ideal
  $\fgl^M_{\aleph_0}=V_*\otimes V\subset \fgl^M$: the former result handles the case of
  countable-dimensional $V$ and $V_*$, whereas the latter result transports
  this to the general case via a categorical equivalence. In
  conclusion, $V_{\mu,\nu}$ is also simple over $\fgl^M$.

  As for the statement regarding the endomorphism algebra, we can
  again assume that we are in the countable-dimensional setup of
  \cite{PS}, as the general case follows then by \cite[Theorem
  5.5]{PS2}.  Then the Lie algebra $V_*\otimes V$ is the union of a chain of
  upper-left-hand-corner inclusions
  \begin{equation*}
    \fgl(2)\subset \fgl(3)\subset \cdots\subset \fgl(q)\subset\cdots,
  \end{equation*}
  and $V_{\mu,\nu}$ is a direct limit of irreducible
  $\fgl(q)$-modules $\left(V_{\mu,\nu}\right)_q$.
  
  Now let $\psi$ be an $\fgl^M$-endomorphism of $V_{\mu,\nu}$.  Consider a vector $0\neq v\in V_{\mu,\nu}$.  Then, $v\in \left(V_{\mu,\nu}\right)_q$ for some $q$.  The vector $\psi(v)$ lies in $\left(V_{\mu,\nu}\right)_{q'}$ for some $q'>q$, hence $\psi(v)$ generates $\left(V_{\mu,\nu}\right)_{q'}$ over $U\big(\fgl(q')\big)$.  Consequently, $\psi|_{\left(V_{\mu,\nu}\right)_{q'}}$ is a well-defined automorphism of $\left(V_{\mu,\nu}\right)_{q'}$, and equals a constant by Schur's Lemma.  The fact that the constants obtained in this way for all $q''>q'$ coincide is obvious.  The statement follows.
\end{proof}

The following result is a direct corollary of \cite[Lemma 3]{Chi14} and \cite[Lemma 3.1]{us}.

\begin{lemma}\label{le.simple_aux}
Let $\fG$ be a Lie algebra and $\mathfrak{I} \subseteq \fG$ {be} an ideal. If $W$ is a $\fG$-module on which $\mathfrak{I} $ acts densely and irreducibly with $\End_\mathfrak{I} (W)=\bK$, then the functor 
\begin{equation*}
  \bullet\otimes W: {\fG/\mathfrak{I}\textup{-mod}}\to {\fG\textup{-mod}}
\end{equation*}
is fully faithful and preserves simplicity. 
  \qedhere
\end{lemma}

We will also need the following auxiliary Schur-Weyl-type result.

\begin{proposition}\label{le.sw}
  Let $W$ be a vector space and $\fG\subseteq \End(W)$ {be} a Lie
  subalgebra { which acts} densely on $W$. Then, for any partition
  $\lambda$, the $\fG$-module $W_\lambda$ is simple and
  $\End_\fG(W_\lambda)\cong \bK$.
\end{proposition}
\begin{proof}  As we noted in Section~\ref{se.prel}, $\mathfrak{G}$ acts densely of $W_\lambda$.  Moreover, it suffices to prove the statement for the case $\mathfrak{G}=\End(W)$, as both simplicity and the endomorphism ring are
  preserved by passing to a Lie subalgebra acting densely (see
  \cite[Lemma 7.3, Theorem 7.4]{PS2}).
The simplicity of $W_\lambda$ as $\End(W)$-module is obvious as $W_\lambda$ is the direct limit of all subspaces $F_\lambda$ for finite-dimensional subspaces $F\subset W$, and the latter are simple $\End(F)$-modules.

Let now $\psi:W_\lambda\to W_\lambda$ be an automorphism.  Choose a
decomposition $W=F\oplus \bar{F}$ where $\dim F<\infty$.  Note that,
for large enough $\dim F$, the $\End(F)$-module $F_\lambda$ is a
submodule of $W_\lambda\Big|_{\End(F)}$ of multiplicity $1$.  Hence,
$\psi\big|_{\End(F)}$ is well defined and $\psi\big|_{F_\lambda}=c$
for some $c\in \bK$.  The fact that $c$ does not depend on the choice
of the decomposition $F\oplus \bar{F}$ follows from the fact that, for
any two decompositions $W=F'\oplus \bar{F}'=F''\oplus \bar{F}''$,
there is a decomposition $W=F'''\oplus \bar{F}'''$ with
$F',F''\subset F'''$.
\end{proof}

The next two results highlight the relevance of \Cref{le.sw} to our setup.

\begin{lemma}\label{le.dense}
  The Lie algebra $\fgl^M/\fgl^M_\alpha$ acts densely on the quotient $V^*/V^*_\alpha$.  
\end{lemma}
\begin{proof}
Since
\begin{equation*}
  \fgl^M_\alpha\cdot V^*\subseteq V^*_\alpha,
\end{equation*}
the Lie algebra ${\fgl^M/\fgl_\alpha^M}$ does indeed act on the
quotient $V^*/V^*_\alpha$. We will henceforth focus on showing that $\fgl^M$ acts
densely. For this purpose, let $v_s$ for $0\le s\le q$ be linearly
independent vectors in $V^*/V^*_\alpha$, and $w_s\in V^*/V_\alpha^*$ be $q$ other
vectors. We have to show that there exists $g\in\fgl^M$ such that
\begin{equation}\label{eq:le.dense}
  g\cdot{\tilde{v}_s=\tilde{w}_s}\text{ for any }s{,}
\end{equation}
{where tilde indicates a preimage in $V^*$}. 

We think of $\tilde{v}_s$ and $\tilde{w}_s$ as row vectors.  The
coordinates of row vectors in $V^*$ are indexed by a totally ordered set $\Sigma$ of
cardinality $\alpha$.  For the duration of the proof, we identify $\Sigma$ with the well-ordered set of all
ordinals $b$ such that $b<\alpha$. The matrices in $\fgl^M$
act on row vectors in $V^*$ as $-vg$, where $vg$ is the product of
matrices. The linear independence of the vectors
$v_s\in V^*/V^*_\alpha$ ensures that their representatives
$\tilde{v}_s$ in $V^*$ have {at least} $\alpha$ nonzero entries.  This
implies that we can partition the set $\Sigma$ of indices of
cardinality $\alpha$ into finite sets $\Sigma_b$ parametrized by
all ordinals $b<\alpha$, and such that each finite set
$\left\{\tilde{v}_s|_{\Sigma_b}\right\}$ is a linearly independent
set of finite vectors; here $\tilde{v}_s|_{\Sigma_b}$ denotes the
finite vector formed by making all entries of $\tilde{v}_s$ outside of
$\Sigma_b$ equal to zero.  

The latter claim is proved by transfinite
induction.  We start by finding $\Sigma_0$ corresponding to the
ordinal $0$: a set $\Sigma_0$ exists such that
$\left\{\tilde{v}_s|_{\Sigma_0}\right\}$ is a linearly independent
set, otherwise the images $v_s$ of $\tilde{v}_s$ in $V^*/V_\alpha^*$
will be linearly dependent.  The transfinite induction step is carried
out in the same way: Let
$\Sigma':=\Sigma\setminus\bigsqcup_{b'<b}\,\Sigma_{b'}$
for some ordinal $b<\alpha$.  If the vectors
$\left\{\tilde{v}_s|_{\Sigma'_b}\right\}$ are linearly dependent
for all choices of a finite set $\Sigma'_b\subseteq \Sigma'$, then
their images in $V^*/V_\alpha^*$ are linearly dependent, a
contradiction.

 The linear independence of the
vectors $\tilde{v}_s|_{\Sigma_b}$ means that we can select column vectors
$g_b$ with finitely many nonzero entries indexed by $\Sigma_b$, and
such that for the product of matrices $-(\tilde{v}_s|_{\Sigma_b}) g_b$ we have
\begin{equation*}
  -(\tilde{v}_s|_{\Sigma_b}) g_b = \text{ the }b\text{-indexed entry of }\tilde{w}_s\text{ for } 0\le s\le q{.}
\end{equation*}
Now simply
take $g$ to be the matrix having the $g_b$ as its columns. It has
finite rows and columns by construction, and satisfies the desired
condition \Cref{eq:le.dense}.
\end{proof}

We can generalize \Cref{le.dense} as follows.

\begin{lemma}\label{le.dense_gen}
  For every infinite cardinal number $\beta\le \alpha$ the quotient
  $\fgl^M_{\beta^+}/\fgl^M_\beta$ acts densely on
  $V^*_{\beta^+}/V^*_\beta$.
\end{lemma}
\begin{proof}
  The case $\beta=\alpha$ is treated in Lemma \ref{le.dense}.  Assume $\beta<\alpha$. We have to show that for any choice of finitely many linearly independent vectors
  $v_s\in V^*_{\beta^+}/V^*_\beta$ and any choice of
  $w_s$ in the same vector space, there exists $g\in \fgl^M$ such
  that
  \begin{equation*}
    g\cdot v_s = w_s\text{ for all }s.
  \end{equation*}
  We can lift $v_s$ and $w_s$ to row vectors in $V^*$ with
  $\beta$ nonzero entries. Having done so, denote by $\Sigma'$ the union of
  the sets of indices of nonzero entries of all these lifted vectors. We can now restrict our
  attention to only those vectors in $V$ and $V^*$ and matrices in
  $\fgl^M$ whose nonzero coordinates have indices in $\Sigma'$.

  This is equivalent to working with the pairing between the
  $\beta$-dimensional subspace $V_{\Sigma'}\subseteq V$ spanned by $\Sigma'$-entry
  vectors and the subspace $(V_{\Sigma'})_*\subseteq V_*$, and
  with the corresponding Mackey Lie algebra. To complete the proof, we simply apply
  \Cref{le.dense} to this pairing of lower-dimensional vector spaces.
\end{proof}

\begin{lemma}\label{le.lambdas}
  For any cardinal number $\beta\le \alpha$ and any partition
  $\lambda$, the module $(V^*_{\beta^+}/V^*_\beta)_\lambda$ is
  irreducible over $\fgl^M_{\beta^+}/\fgl^M_\beta$, and its
  endomorphism ring is $\bK$.
\end{lemma}
\begin{proof}
  This is an immediate application of \Cref{le.sw} and \Cref{le.dense_gen}.
\end{proof}

\begin{proof-of-simples}
We split the proof into two portions.

{\bf Part 1: Simplicity and endomorphism algebra of the object \Cref{eq:tensor}.} We
prove this by induction on $t$, the case $t=0$ being a consequence of \cite[Theorem 4.1]{PS2}.

Now assume that the statement holds for $t-1$ and set $\beta=\beta_t$
and $\lambda=\lambda_t$. \Cref{le.dense_gen} and \Cref{le.sw} ensure that the tensorand
$(V^*_{\beta^+}/V^*_\beta)_\lambda$ of \Cref{eq:tensor} is
simple over $\fgl^M_{\beta^+}/\fgl^M_{\beta}$ with scalar endomorphism
algebra. Setting
\begin{equation*}
  W={ V_{(\beta_{t-1},\lambda_{t-1}),\cdots,(\beta_{0},\lambda_{0}),\mu,\nu}},
\end{equation*}
we can then apply \Cref{le.simple_aux} to the ideal
$\fgl^M_\beta\subseteq \fgl^M_{\beta^+}$ (in the role of
$\mathfrak{I} \subseteq \fG$) to finish the proof.

{\bf Part 2: The simple objects are mutually nonisomorphic.} Suppose that
the modules
{$ V_{(\beta_t,\lambda_t),\cdots,(\beta_{0},\lambda_{0}),\mu,\nu}$ and
  $V_{(\beta'_q,\lambda'_q),\cdots,(\beta'_{0},\lambda'_{0}),\mu',\nu'}$
  are isomorphic}. Restricting first to $\fgl^M_{\aleph_0}=V_*\otimes V$, over which the
modules are direct sums of copies of $V_{\mu,\nu}$ and $V_{\mu',\nu'}$
respectively, we get $\mu'=\mu$ and $\nu'=\nu$. We can now proceed
recursively in the following fashion.

Assume that for some $u\le\min (q,t)$ we have shown that
\begin{equation*}
  \beta_s=\beta'_s\text{ and } \lambda_s=\lambda'_s\text{ for } 0\le s\le u.
\end{equation*}
Then, setting 
\begin{equation*}
   W={ V_{(\beta_u,\lambda_u),\cdots,(\beta_{0},\lambda_{0}),\mu,\nu}}
\end{equation*}
and 
\begin{equation*}
  \mathfrak{I} \subseteq \fG \text{ to be } \fgl^M_{\beta_{u+1}}\subseteq \fgl^M\,,
\end{equation*}
 we conclude from \Cref{le.simple_aux} that 
 \begin{equation*}
   \bigotimes_{s=t}^{u+1}(V^*_{\beta_s^+}/V^*_{\beta_s})_{\lambda_s} \cong  \bigotimes_{s=q}^{u+1}(V^*_{\beta'^+_s}/V^*_{\beta'_s})_{\lambda'_s}.
 \end{equation*}
 Restricting this isomorphism to $\fgl^M_{\beta^+}/\fgl^M_\beta$ for
 $\beta=\min(\beta_{u+1},\beta'_{u+1})$ we conclude that
 $\beta_{u+1}=\beta'_{u+1}$ and $\lambda_{u+1}=\lambda'_{u+1}$. We can
 now repeat the procedure with $u$ in place of $u+1$, until the
 process terminates. This can only happen if $q=t$ and the
 corresponding $\beta_s$ and $\beta'_s$ are equal, and similarly for
 $\lambda_s$ and $\lambda'_s$.
\end{proof-of-simples}

\subsection{Ordering $\bar{\bT}_\alpha$}\label{subse.order}

Here we explain how the category $\bar{\bT}_\alpha$ fits
into the setting of \Cref{se.abs}.

Our objects $X_i$ will be {finite} tensor products of the form
\begin{equation}\label{eq:xi}
  { \left(\bigotimes_{\beta} \,\left(V^*/V^*_{\beta}\right)^{\otimes n_\beta}\right)\otimes (V^*)^{\otimes n}\otimes V^{\otimes m}}
\end{equation}
for infinite cardinal numbers $\beta\le\alpha$.  {In this way,} the
underlying set $I$ of the poset indexing the objects $X_i$ consists of
all finite tuples
\begin{equation*}
  (n_\beta,n,{m})_{\beta\le \alpha}
\end{equation*}
of nonnegative integers where almost all $n_\beta$ vanish.  We define a partial order on $I$ by setting
\begin{equation*}
  (n_\beta,n,{m}) {\preceq} (n'_\beta,n',{m'})
\end{equation*}
{if and only the following conditions hold:}
\begin{enumerate}
  \renewcommand{\labelenumi}{(\alph{enumi})}
    \item if $\beta$ is the largest cardinal {with} $n_\beta\ne n_\beta'$ then $n_\beta>n_\beta'$;
    \item $n\le n'$  and {$m\le m'$};
    \item {$\sum_\beta n_\beta+n-m = \sum_\beta n_\beta'+n'-m'$}.
\end{enumerate}

It is easy to check that the opposite poset $(I,\preceq)^{\text{op}}$ is well-ordered.   In order to show that the above choice of objects $X_i$ for $i\in I$ makes $\bar{\bT}_\alpha$ an ordered Grothendieck category, we start with

\begin{lemma}\label{le.ess}
  Let $\beta_0<\cdots<\beta_t\le \alpha$ be infinite cardinal numbers,
  and $\lambda_t,\ldots,\lambda_0,\mu,\nu$ be arbitrary Young diagrams.

  Then, the object of $\bT_\alpha$
  \begin{equation}\label{eq:injective}
    \bigotimes_{s=t}^0 (V^*/V^*_{\beta_s})_{\lambda_s}\otimes (V^*)_{\mu}\otimes V_{\nu}
  \end{equation}
is an essential extension of   
  \begin{equation}\label{eq:simple}
  V_{(\beta_t,\lambda_t),\cdots,(\beta_0,\lambda_0),\mu,\nu}=\bigotimes_{s=t}^0 (V^*_{\beta_s^+}/V^*_{\beta_s})_{\lambda_s}\otimes V_{\mu,\nu}.
\end{equation}
\end{lemma}
\begin{proof}

  {\bf Step 1.} In first instance we argue that \Cref{eq:simple} is essential
  in
  \begin{equation}\label{eq:simple'}
    \bigotimes_{s=t}^0 (V^*_{\beta_s^+}/V^*_{\beta_s})_{\lambda_s}\otimes  (V^*)_{\mu}\otimes V_{\nu}.
  \end{equation}
To this end, note first that it suffices to show that \Cref{eq:simple}
is essential in \Cref{eq:simple'} when regarded as a module over the ideal $V_*\otimes V$ of $\fgl^M$. Since $V_*\otimes V$ annihilates $ \bigotimes_{s=t}^0 (V^*_{\beta_s^+}/V^*_{\beta_s})_{\lambda_s}$, this in turn
reduces to showing that $V_{\mu,\nu}$ is essential in
$(V^*)_\mu\otimes V_\nu$ as a module over $V_*\otimes V$.

The inclusions
  \begin{equation*}
    (V_*)^{\otimes n}\otimes V^{\otimes m}\subset (V^*)^{\otimes n}\otimes V^{\otimes m}
  \end{equation*}
  are essential: for any $v\ne 0$ belonging to the right-hand side, the
  $V_*\otimes V$-module generated by $v$ is nonzero and contained
  in the left-hand side.

  Now simply apply this remark to a traceless $v$ in the image of the
  Young symmetrization operator sending
  $(V^*)^{\otimes n}\otimes V^{\otimes m}$ to
  $(V^*)_\mu\otimes V_\nu$.

  This concludes the proof that \Cref{eq:simple} is
  essential in \Cref{eq:simple'}.

  \vspace{.5cm}

  {\bf Step 2.} We next argue that \Cref{eq:simple'} is essential in
  \begin{equation}\label{eq:simple''}
    \bigotimes_{s=t}^1 (V^*_{\beta_s^+}/V^*_{\beta_s})_{\lambda_s}\otimes (V^*/V^*_{\beta_0})_{\lambda_0}\otimes  (V^*)_{\mu}\otimes V_{\nu}.
  \end{equation}
  This is very similar in spirit to the proof of Step 1: it is enough
  to prove that the extension in question is essential over the Lie
  subalgebra $\fgl^M_{\beta_0^{+}}\subseteq \fgl^M$ which annihilates
  the $s$-indexed tensorands in \Cref{eq:simple''} and maps the middle
  tensorand $(V^*/V^*_{\beta_0})_{\lambda_0}$ into
  $(V^*_{\beta_0^+}/V^*_{\beta_0})_{\lambda_0}$.

  As before, it suffices to observe that $\fgl^M_{\beta_0^+}$ does not
  annihilate any nonzero elements of
  $(V^*/V^*_{\beta_0})_{\lambda_0}\otimes (V^*)_{\mu}\otimes V_{\nu}$.

  \vspace{.5cm}

  {\bf Step 3: conclusion.} We now repeat the argument in Step 2
  inductively, each time replacing one $V^*_{\beta_s^+}$ by $V^*$ and
  working over the ever-larger Lie algebra
  $\fgl^M_{\beta_s^{+}}$.  After exhausting all tensorands, we obtain a tower of essential extensions with the simple object (\ref{eq:injective}) at the bottom and the object (\ref{eq:xi}) at the top.  The desired conclusion follows.
\end{proof}

%

\begin{corollary}
	The transfinite filtration (\ref{eq.transfilt}) is the transfinite socle filtration of the object $V^*$ of $\bT_\alpha$.
\end{corollary}

\begin{proof}
	This follows immediately from the simplicity of the objects $V^*_{\beta^+}/V^*_\beta$ for cardinals $\beta\leq \alpha$ and from the fact that $V^*/V^*_\beta$ is an essential extension of $V^*_{\beta^+}/V^*_\beta$.
\end{proof}

\begin{proposition}\label{pr.isordered} {For the above choice of the
    poset $I$ and the objects $X_i$, for $i\in I$,
    the Grothendieck envelope $\bar{\bT}_\alpha$ is is an ordered
    Grothendieck category.} 
\end{proposition}
\begin{proof}
  Part (i) of \Cref{def.ordered} is implicit in \Cref{def.t}.  For part (ii), note that $X_i$ is a direct sum of objects of the form (\ref{eq:injective}).  Since (\ref{eq:injective}) is an essential extension of (\ref{eq:simple}) by Lemma \ref{le.ess}, and the object (\ref{eq:simple}) is simple by Proposition~\ref{pr.simples}, $\soc X_i$ is a direct sum of objects (\ref{eq:simple}) with $\left|\lambda_\beta\right|=n_\beta$, $|\mu|=n$, $|\nu|=m$.  This implies (ii).  Now, part (iv) is also clear as the object (\ref{eq:injective}) is indecomposable with simple socle (\ref{eq:simple}). 
  
  For part (iii), one has to check that any simple subquotient of an object (\ref{eq:injective}) satisfies conditions a), b), and c) from the definition of the partial order $\preceq$ on $I$.  This is straightforward if we note that the injective object \Cref{eq:injective} has a filtration obtained by tensoring the filtration of
$(V_*)_\mu\otimes V_\nu$ from \cite[Theorem 2.2]{PS} and those of the
various $V^*/V^*_{\beta_s}$ with simple subquotients of the form
$V^*_{\beta^+_r}/V^*_{\beta_r}$.   We leave the details to the reader.  See also Example \ref{ex.examples} below for a particular case where all simple subquotients of (\ref{eq:injective}) are displayed explicitly.
%
%

  Finally, for part (v), it is a rather routine verification that having
  fixed $i$ and $j$ in $I$ as in that portion of \Cref{def.ordered},
  the morphisms $X_i\to X_j$ obtained by composing and tensoring
  surjections $V^*/V^*_\beta\to V^*/V^*_\gamma$ for $\gamma\ge \beta$
  and contractions $V^*\otimes V\to \bC$ will satisfy the condition.
\end{proof}

Our next result classifies the indecomposable injective objects of
$\overline{\bT}_\alpha${; these happen to already be contained in $\bT_\alpha$}.

\begin{theorem}\label{th.indec_inj-gen}
  The indecomposable injectives in the category
  $\overline{\bT}_\alpha$ are, up to isomorphism, the objects
  \Cref{eq:injective} with respective socles \Cref{eq:simple}, where
  the choices range over tuples of infinite cardinal numbers
  $\beta_0<\cdots<\beta_t\le \alpha$ and Young diagrams
  $\lambda_t,\ldots,\lambda_0,\mu,\nu$.
\end{theorem}
\begin{proof}
  This is a consequence of \Cref{pr.isordered} and \Cref{cor.indec-inj} together
  with the classification of simple objects from \Cref{pr.simples} and
  the fact that \Cref{eq:simple} is essential in \Cref{eq:injective}
  via \Cref{le.ess}.
\end{proof}

\subsection{Blocks of $\bar{\bT}_\alpha$}
\label{sec.blocksbT}

We will show that the integers appearing on the two sides of
the equality (c) in the definition of $(I,{\preceq})$ in Section \ref{subse.order} in fact
parametrize the blocks of the category $\overline{\bT}_\alpha$.  First, let us recall

\begin{definition}\label{def.blocks}
  Suppose the class $\mathrm{Indec}(\cC)$ of isomorphism classes of indecomposable objects of
  an abelian category $\cC$ is a set. The {\it blocks} of $\cC$ are
  the classes of the finest equivalence relation on
  $\mathrm{Indec}(\cC)$ requiring that objects $Z,Y$ with
  $\Hom_\cC(Z,Y)\ne 0$ belong to the same class.
\end{definition}

\begin{theorem}\label{th.blocks}
  The blocks of $\overline{\bT}_\alpha$ are parametrized
  by $\bZ$, and the simple object 
    {$V_{(\beta_t,\lambda_t),\cdots,(\beta_0,\lambda_0),\mu,\nu}$}
  belongs to the block indexed by $\sum\limits_{s=t}^0|\lambda_s|+|\mu|-|\nu|$.
\end{theorem}
\begin{proof}
For a simple object $
  U=V_{(\beta_t,\lambda_t),\cdots,(\beta_0,\lambda_0),\mu,\nu}
$
as in the statement, let us refer to the integer 
\begin{equation*}
  \sum\limits_{s=t}^0|\lambda_s|+|\mu|-|\nu|\in \bZ
\end{equation*}
as the {\it content} of $U$. We split the proof into two halves.

\vspace{.5cm}

{\bf (1) Different content $\Rightarrow$ different blocks.} Since the content of a simple object $U$ is nothing but the expression in part c) of the definition of the partial order on $I$ (where $n_\beta=\left|\lambda_\beta\right|$, $n=|\mu|$, $m=|\nu|$), the fact that $\bar{\bT}_\alpha$ is an ordered Grothendieck category (Proposition \ref{pr.isordered}) implies that 
all
simple subquotients of the injective hull \Cref{eq:injective} of a
simple object $U$ as in \Cref{eq:simple} have the same content as $U$.

\vspace{.5cm}

{\bf (2) Same content $\Rightarrow$ same block.} Consider a simple
object $U$ of the form \Cref{eq:simple}. Its injective hull
\Cref{eq:injective} surjects onto
\begin{equation*}
  \bigotimes_{s=t}^0 (V^*/V^*_{\beta_t})_{\lambda_s} \otimes (V^*/V^*_{\beta_t})_{\mu}\otimes V_\nu,
\end{equation*}
so $U$ is in the same block as an injective object of the form
\begin{equation}
  (V^*/V^*_{\beta_t})_{\lambda}\otimes V_\nu,
  \label{eq.someobject}
\end{equation}
for some Young diagram $\lambda$.  In turn, the
object~(\ref{eq.someobject}) is a quotient of the indecomposable
injective $\left(V^*\right)_\lambda\otimes V_\nu$. 
Finally, the classification of blocks in $\bT_{\fsl(\infty)}$ from \cite[Corollary 6.6]{DPS}, together with the equivalence of categories $$\bT_{\End(V)}\simeq \bT_{\fsl(\infty)}$$ established in \cite{PS2}, shows that the block of $\left(V^*\right)_\lambda \otimes V_\nu$ depends only on the difference $|\lambda|-|\nu|$.  The result follows.
\end{proof}

\subsection{Vanishing Ext-functors and Koszulity}\label{se.exts}

Our main aim in the present subsection is to prove an analogue of
\cite[Theorem 3.11]{us}, improving on \Cref{pr.up-tri} and describing necessary
conditions for nonvanishing Ext-functors between the simple
objects of $\bT_\alpha$ described in \Cref{pr.simples}.


We start with a formula for the defect $d(i,j)$.  Consider two
elements $i$ and $j$ of $I$: $i=\left(n_\beta,n,m\right)$ and
$j=\left(n_\beta',n',m'\right)$.  Let
$\left\{\beta'_0<\beta'_1<\ldots<\beta'_{q'}\right\}$ be the union of
all infinite cardinals for which $n_\beta\neq 0$ or
$n'_{\beta'}\neq 0$.  Extend this set to a minimal set of cardinals
which is \emph{interval-closed} in the sense that whenever
$\beta_{j'}$ is a finite iterated successor $\beta_j^{++\cdots +}$,
all intermediate successors $\beta_j^+$, $\beta_j^{++}$, etc. belong
to the set.  Denote by $\left\{\beta_0<\ldots<\beta_q\right\}$ the
resulting finite set.

\begin{proposition}\label{pr.form}
Given $i=(n_\beta,n,m)\preceq j= (n'_\beta,n',m')$ with
finite defect $d(i,j)$, we have
\begin{equation}
d(i,j)=
n'-n+\sum_{s=0}^q\,s\left(n_s-n'_s\right)\,.
\label{eq.defect}
\end{equation}
\end{proposition}
\begin{proof}
  We argue by induction on $d(i,j)$.  If {$d(i,j)=1$}, then there are two possibilities:
  \begin{itemize}
  \item $n_\beta=n'_\beta$ for all cardinals $\beta$ and
        {$n'-n = m'-m = 1$}
  \end{itemize}
      or
  \begin{itemize}
  \item {$m'=m$, $n'=n$, $n'_{\beta}\neq n_{\beta}$ for precisely two cardinals $\beta$ of the form $\beta_s$, $\beta_{s+1}=\beta_s^+$, and $
        n'_{\beta_s}-n_{\beta_s} = n_{\beta_{s+1}}-n'_{\beta_{s+1}} = 1$.}
  \end{itemize}
    In both cases formula (\ref{eq.defect}) holds.  For the induction step, consider a maximal chain
\begin{equation*}
  i=i_0\prec i_1\cdots\prec i_{u}=j
\end{equation*}
for $u=d(i,j)$.  Then $d\left(i_{u-1},i_u\right)=1$, and using the above observation and the induction assumption, one immediately checks formula (\ref{eq.defect}) for the pair $(i,j)$.
\end{proof}

\begin{corollary}\label{cor.defecttriangle}
\begin{enumerate}[(a)]
\item If $d(i,j)<\infty$, the length of any finite chain in $I$ linking $i$ and $j$ equals $d(i,j)$.
\item Under the assumption of (a), we have
  \begin{equation*}
    d(i,k) = d(i,j)+d(j,k)
  \end{equation*}
\end{enumerate}
\end{corollary}
\begin{proof}
  Part (a) emerged over the course of the proof of \Cref{pr.form},
  while part (b) is a consequence of (a): if all maximal chains have
  the same length, then that common length must be equal to the length
  of a chain $i\to k$ obtained by concatenating chains $i\to j$ and
  $j\to k$.
\end{proof}

\begin{remark}\label{re.rank}  
  In the language of \cite[Section 3.1]{ec1}, intervals $[i,j]$ for which
  $d(i,j)<\infty$ are {\it graded} posets (some authors refer to
  such posets as {\it ranked}).
\end{remark}

\begin{notation}\label{not.def-obj}
  For simple objects $U\in \cS_i$ and $W\in \cS_j$ with $i\preceq j$
  we sometimes write $d(U,W)$ for
  $d(i,j)$.
\end{notation}


We will use the following observations in the proof of \Cref{th.exts}
{below}.

\begin{corollary}\label{re.defect-ex}
  Suppose $U\in \mathcal{S}_i$ and $W \in \mathcal{S}_j$ are simple objects {of $\bT_\alpha$ such that}
  $d(i,j)=p$.Then for every infinite cardinal $\gamma$ and any simple
  objects
  \begin{equation*}
    \widetilde{U}\subset \soc((V^*/V_\gamma^*)\otimes U),\ \widetilde{W}\subset \soc((V^*/V_\gamma^*)\otimes W)
  \end{equation*}
  we have $d(\widetilde{U},\widetilde{W})=p$. Similarly, for a simple object
  \begin{equation*}
    \widetilde{\widetilde{U}}\subset \soc((V^*/V_{\gamma^+}^*)\otimes U)
  \end{equation*}
  we have $d\left(\widetilde{\widetilde{U}},\widetilde{W}\right)=p+1$.
\end{corollary}
\begin{proof}
  Let $i=\left(n_\beta,n,m\right)$, $j=\left(n'_\beta,n',m'\right)$.
  Then $\widetilde{U}\in \mathcal{S}_{\tilde{i}}$, for
  $\tilde{i}=\left(\tilde{n}_\beta,n,m\right)$, where
  $\tilde{n}_\beta=n_\beta$ for $\beta\neq \gamma$, and
  $\tilde{n}_\gamma=n_\gamma+1$.  Similarly,
  $\widetilde{W}\in \mathcal{S}_{\tilde{j}}$ for
  $\tilde{j}=\left(\tilde{n}'_\beta,n',m'\right)$, where
  $\tilde{n}'_\beta=n'_\beta$ for $\beta\neq \gamma$, and
  $\tilde{n}'_\gamma=n'_\gamma+1$.  Finally,
  $\widetilde{\widetilde{U}}\in \mathcal{S}_{\tilde{\tilde{i}}}$, for
  $\tilde{\tilde{i}}=\left(\tilde{\tilde{n}}_\beta,n,m\right)$, where
  $\tilde{\tilde{n}}_\beta=n_\beta$ for $\beta\neq \gamma^+$,
  and $\tilde{\tilde{n}}_{\gamma^+}=n_{\gamma^+}+1$.
  Therefore both claims follow from immediate application of Proposition \ref{pr.form}.
%
%
%
\end{proof}


%

The first main result of the present subsection is

\begin{theorem}\label{th.exts}
  Let $T'\in \cS_i$ and $T\in \cS_j$ be two simple objects in
  $\bT_\alpha$, and suppose $\ext^p(T',T)\ne 0$ for some $p\ge
  0$. Then $d(i,j)=p$.
\end{theorem}

Before beginning the proof, we introduce

\begin{notation}\label{not.ni}
  For $i=(n_\beta,n,m)\in I$ as in \Cref{subse.order} we denote
  \begin{equation*}
    {m_i = m}, n_i=n,\quad n_{\beta,i} = n_\beta. 
  \end{equation*}
  when we wish to extract the components $n_\beta$, $n$, and $m$ from $i$.
\end{notation}

\begin{proof-of-exts}
  We already know from \Cref{pr.up-tri} that $\ext^p(T',T)\neq 0$ implies $i {\preceq} j$ and
  $d(i,j)\ge p$; the inequality $i\preceq j$ will be implicit throughout the proof.

  We do simultaneous induction on $p$ and the nonnegative integer
\begin{equation}\label{eq:ind-number}
  N_j:=\sum_{\beta}n_{\beta,j}
\end{equation}
(see \Cref{not.ni}). For $p=0$ the statement is trivial as
there are no nonzero morphisms between distinct irreducible
objects. Similarly, the case $N_j=0$ is also immediate, since $T$ will
then be both simple and injective by \Cref{th.indec_inj-gen}.

We now address the induction step, assuming that both $p$ and $N_j$
are strictly positive. In that case, we can find some cardinal
$\beta$, which is infinite or equals zero, such that $T$ is a direct
summand of the socle of $(V^*/V_\beta^*)\otimes U$ for some simple
object $U$.   By \Cref{le.ess} the socle of $(V^*/V_\beta^*)\otimes U$ is
$(V^*_{\beta^+}/V^*_\beta)\otimes U$, and hence $T$ is a direct summand of
the latter.  Consider the short
exact sequence
\begin{equation}\label{eq:seq}
  0\to (V^*_{\beta^+}/V^*_\beta)\otimes U \to (V^*/V_\beta^*)\otimes U\to (V^*/V_{\beta^+}^*)\otimes U\to 0,
\end{equation}
in which we set $V_\beta^*=0$, $V_{\beta^+}^*=V_*$ for $\beta=0$.

The assumption $\ext^p(T',T)\ne 0$ implies
\begin{equation*}
  \ext^p(T',(V^*_{\beta^+}/V^*_\beta)\otimes U)\ne 0,
\end{equation*}
and via the long exact sequence for Ext-groups, this entails
\begin{equation}\label{eq:p}
  \ext^p(T',(V^*/V_\beta^*)\otimes U)\ne 0
\end{equation}
or 
\begin{equation}\label{eq:p-1}
  \ext^{p-1}(T',(V^*/V_{\beta^+}^*)\otimes U)\ne 0.
\end{equation}

Consider first the case (\ref{eq:p}).  
\Cref{th.indec_inj-gen} makes it clear that tensor products of
injective objects with finite-length socle are again injective, and
hence tensoring an injective resolution of $U$
\begin{equation*}
  0\to U\to J^0\to J^1\to\cdots
\end{equation*}
with $V^*/V^*_\beta$ produces an injective resolution $(V^*/V_\beta^*)\otimes J^\bullet$ of $(V^*/V^*_\beta)\otimes U$.  Since $U\in\mathcal{S}_{j'}$ with $N_{j'}=N_j-1$, the induction hypothesis ensures that the socle of $J^p$ consists of
simple objects $W$ in various $\cS_t$ with $d(W,U)=p$.  Therefore, by Corollary \ref{re.defect-ex}, $d\left(W',T\right)=p$ for all simple objects $W'$ in $\soc (V^*/V^*_\beta \otimes J^p)$, and the claim follows.

In the case of (\ref{eq:p-1}), the proof is very similar an uses the second statement of Corollary \ref{re.defect-ex}. 
\end{proof-of-exts}

We next use \Cref{th.exts} to show that the Grothendieck category $\overline{\bT}_\alpha$ is
Koszul according to the following
definition. 

\begin{definition}\label{def.kosz-cat}
  Let $\cC$ be a semi-artinian Grothendieck category such that the
  class of isomorphism classes of simple objects is a set. We denote
  by $E\cC$ the {\it Ext-algebra} of $\cC$, by definition equal to
  \begin{equation*}
    \bigoplus \ext^p(T',T),
  \end{equation*}
  made into an algebra via the Yoneda product; the summation ranges over the isomorphism classes of simple objects $T$, $T'$, and integers $p\geq 0$. Note that $E\cC$ is
  graded by placing $\ext^p$-groups in degree $p$.

  The category $\cC$ is said to be {\it Koszul} if $E\cC$ is generated in degree
  one.
\end{definition}

We can now state

\begin{theorem}\label{th.kosz-gen}
  The category $\overline{\bT}_\alpha$ is Koszul for any infinite cardinal $\alpha$.
\end{theorem}
\begin{proof}
  We have to prove that for each nonnegative integer $p$ the
  degree-$p$ component of $E\overline{\bT}_\alpha$ is contained in the
  span of Yoneda products of elements in the degree-$1$ and
  degree-$(p-1)$ components.
  
  Let $T'$ and $T$ be simple objects in $\bT_\alpha$, with $d(T',T)=p$
  (see \Cref{not.def-obj}), and let $\widetilde{T}$ be an injective
  hull of $T$.  Consider a nonzero element $x\in \ext^p(T',T)$. It is obtained,
  via the long exact sequence of Exts corresponding to the exact sequence
  \begin{equation*}
    0\to T\to \widetilde{T} \to \cor T\to 0,
  \end{equation*}
  determined by a nonzero element of $\ext^{p-1}(T',\cor T)$. In turn,
  \Cref{th.exts} implies that this is actually a nonzero element of
  \begin{equation}\label{eq:s'u}
    \ext^{p-1}\big(T',\soc(\cor T)\big)\cong \bigoplus \ext^{p-1}(T',U)\,,
  \end{equation}
  where the summation ranges over the isomorphism classes of the simple submodules $U$ of $\cor T$.
  Finally, this means that $x$ is in the span of the Yoneda product of the space
  \Cref{eq:s'u} and $\bigoplus \ext^1(U,T)$.   These spaces are subspaces of
  $(E\overline{\bT}_\alpha)_{p-1}$ and $(E\overline{\bT}_\alpha)_1$
  respectively.
\end{proof}

\subsection{The case $\alpha=\aleph_t$ for $t\in\mathbb{Z}_{\geq 0}$}\label{subse.fin}

In this section, $\alpha=\aleph_t$ for $t\in\bZ_{\geq 0}$.

\begin{proposition}\label{prop.finordcat}
Under the above assumption, the objects $X_i$ have finite length, and thus $\bar{\bT}_{\alpha}$ is a finite ordered Grothendieck category.
\end{proposition}

\begin{proof}
	Since $V^*$ has finite length, the object $\left(V^*\right)^{\otimes n}$ has a finite filtration whose successive quotients are direct sums of modules of the form $$\left(V^*/V^*_{\aleph_t}\right)^{\otimes q_t}\otimes \cdots \otimes \left(V^*_{\aleph_{1}} / V^*_{\aleph_0}\right)^{\otimes q_0}\otimes \left(V_*\right)^{\otimes q} \otimes V^{\otimes q'}\,.$$	
By \cite{PS2} and \cite{PS}, $\left(V_*\right)^{\otimes q} \otimes V^{\otimes q'}$ has finite length with irreducible subquotients of the form $V_{\mu,\nu}$.  Furthermore, modules of the form $\left(V^*/V^*_{\aleph_t}\right)^{\otimes q_t}\otimes \cdots \otimes \left(V^*_{\aleph_{1}} / V^*_{\aleph_0}\right)^{\otimes q_0}$ have finite filtrations with irreducible subquotients of the form $V_{\lambda_t,\ldots,\lambda_0,\emptyset,\emptyset}$ for Young diagrams $\lambda_t,\ldots,\lambda_0$ (some of which may be empty).   Finally, Proposition \ref{pr.simples} implies that tensor products of the form $V_{\lambda_t,\ldots,\lambda_0,\emptyset,\emptyset}\otimes V_{\mu,\nu}$ are irreducible.  The statement follows.
\end{proof}

	Proposition~\ref{prop.finordcat} implies in particular that the discussion in \cite[Section 2]{us} applies verbatim.

%
%

\begin{corollary}\label{cor.bigone}
  \begin{enumerate}[(a)]
  	\item $\bar{\bT}_{\aleph_t}$ is a
  highest weight category.
  	\item The indecomposable injectives in the category $\bar{\bT}_{\al_t}$ are (up to isomorphism)
  \begin{equation}\label{eq:inj}
   \tilde{V}_{\lambda_t,\ldots,\lambda_0,\mu,\nu}=\bigotimes_{s=t}^0(V^*/V^*_{\al_s})_{\lambda_s} \otimes (V^*)_{\mu}\otimes V_{\nu},
  \end{equation}
with respective socles 
  \begin{equation}\label{eq:respsocle}
    V_{\lambda_t,\cdots,\lambda_0,\mu,\nu}=\bigotimes_{s=t}^0(V^*_{\al_{s+1}}/V^*_{\al_s})_{\lambda_s} \otimes V_{\mu,\nu}
  \end{equation}
   for arbitrary Young diagrams $\lambda_t,\ldots,\lambda_0,\mu,\nu$;
  \item The injective objects {of $\bar{\bT}_{\aleph_t}$} are precisely the direct sums
  of indecomposable injectives of the form (\ref{eq:inj}).
  
  \item The category $\overline{\bT}_{\aleph_t}$ is equivalent to {the category} of
  comodules over a Koszul semiperfect coalgebra $C$. Moreover, any
  such equivalence identifies $\bT_{\aleph_t}$ {with} the category of
  finite-dimensional $C$-comodules. 
  \end{enumerate}
\end{corollary}

\begin{proof}
	Part (a) follows from Proposition \Cref{pr.hw}.  Part (b) is a particular instance of Theorem \ref{th.indec_inj-gen}.  For part (c), we note that  \Cref{pr.loc-noe} applies to the
  category $\bar{\bT}_{\aleph_t}$.  Part (d) is a consequence of \Cref{th.kosz-gen} and
\cite[Theorem 2.13]{us} (cf. \cite[Corollary 3.17]{us}).
\end{proof}

Note that the module (\ref{eq:respsocle}) is nothing but the module (\ref{eq:simple}).  When $\alpha=\aleph_t$, the infinite cardinal $\beta_s$ can be read of the position of $\lambda_s$ in the subscript of the left-hand side of (\ref{eq:respsocle}).  Therefore we can simply write $V_{\lambda_t,\ldots,\lambda_0,\mu,\nu}$ instead of $V_{\left(\beta_t,\lambda_t\right),\ldots,\left(\beta_0,\lambda_0\right),\mu,\nu}$.

\begin{remark}\label{re.not-unique}
  The coalgebra $C$ in Corollary \ref{cor.bigone}(d) is not unique, being
  determined only up to Morita equivalence. Nevertheless, for a
  specific choice, see \Cref{re.coalg-choice} below.
\end{remark}

Our next goal is to describe the socle filtration of the injectives (\ref{eq:inj}).  We start with

\begin{lemma}\label{lem:blah}
	Fix $\lambda_t,\ldots,\lambda_0,\mu,\nu$ as in Corollary \ref{cor.bigone}.  Then $V_{\kappa_t,\ldots,\kappa_0,\gamma,\delta}$ is a direct summand of $\usoc^s\widetilde{V}_{\lambda_t,\ldots,\lambda_0,\mu,\nu}$ if and only if $V_{\kappa_t,\ldots,\kappa_0,\gamma,\delta}$ is a constituent of $\tilde{V}_{\lambda_t,\ldots,\lambda_0,\mu,\nu}$ and $d(i,j)=s-1$, where $i=\left(|\kappa_t|,\ldots,|\kappa_0|,|\gamma|,|\delta|\right)$ and $j=\left(|\lambda_t|,\ldots,|\lambda_0|,|\mu|,|\nu|\right)$.
\end{lemma}

\begin{proof}
	Theorem \ref{th.exts} implies that $V_{\kappa_t,\ldots,\kappa_0,\gamma,\delta}$ is a submodule of $\usoc^2 \tilde{V}_{\lambda_t,\ldots,\lambda_0,\mu,\nu}$ if and only if $V_{\kappa_1,\ldots,\kappa_0,\gamma,\delta}$ is a constituent of $\tilde{V}_{\lambda_t,\ldots,\lambda_0,\mu,\nu}$ and $d(i,j)=1$.  The general case follows by a straightforward induction argument.
\end{proof}

We now describe the socle filtration of $\tilde{V}_{\lambda_t,\ldots,\lambda_0,\mu,\nu}$ explicitly.  Our approach is to break down the problem into manageable pieces.


As in \cite[$\S$3.2]{us}, we
use Sweedler notation $\Delta:\lambda\mapsto
\lambda_{(1)}\otimes\lambda_{(2)}$ for the comultiplication in the
Hopf algebra $\sym$ of symmetric functions (\cite{sym}) with the usual
basis consisting of Schur functions indexed by partitions
$\lambda$. This notation propagates to multiple iterations of the
comultiplication, as in 
\begin{equation*}
  (\Delta\otimes\id)\circ\Delta:\lambda\mapsto \lambda_{(1)}\otimes \lambda_{(2)}\otimes\lambda_{(3)},
\end{equation*}
etc. We also reprise the notation for truncating such sums according
to the number of boxes in the partitions indexing the basis
elements. For instance, 
\begin{equation*}
  \lambda_{(1)}^q\otimes\lambda_{(2)}^{|\lambda|-q}
\end{equation*}
denotes the sum of all of those summands in
$\Delta(\lambda)=\lambda_{(1)}\otimes\lambda_{(2)}$ whose left-hand
tensorand corresponds to a partition with $k$ boxes.

Furthermore, in order to shorten the notation, we will sometimes denote the simple object $V^*_{\al_{u+1}}/V^*_{\al_u}$ by $U_u$.

\begin{lemma}\label{le.quot_filt}
  Let $0\leq u\le t$ and let $\lambda$ be a partition. Then $\usoc^q\left((V^*/V^*_{\al_u})_\lambda\right)$
  is isomorphic to
  \begin{equation}\label{eq:directsumU}
    \bigoplus (U_t)_{\lambda_{(t-u+1)}^{\ell_{t-u+1}}}\otimes\cdots\otimes (U_u)_{\lambda_{(1)}^{\ell_1}},
  \end{equation}
where the summation is over all choices of nonnegative integers $\ell_x$ such that
\begin{equation}\label{eq:fats}
  \sum^{1}_{x=t-u+1}\ell_x=|\lambda|
\end{equation} 
and
\begin{equation}\label{eq:fats2}
\sum^{1}_{x=t-u+1}(x-1)\ell_x=q-1.
\end{equation}
\end{lemma}
\begin{proof}
  The fact that the simple modules in the sum (\ref{eq:directsumU}) satisfying (\ref{eq:fats}) are constituents of $\left(V^*/V^*_{\aleph_u}\right)_\lambda$ follows from the general remark that given an exact sequence
  \begin{equation*}
    0\to U\to W\to Y\to 0
  \end{equation*}
  of vector spaces, $W_\lambda$ admits a filtration
  \begin{equation*}
    0\subseteq U_\lambda \subseteq U_{\lambda^{|\lambda|-1}_{(1)}}\otimes Y_{\lambda^{1}_{(2)}} \subseteq\cdots\subseteq U_{\lambda^1_{(1)}}\otimes Y_{\lambda^{|\lambda|-1}_{(2)}} \subseteq Y_\lambda. 
  \end{equation*}
	According to Lemma \ref{lem:blah}, the equality (\ref{eq:fats2}) singles out the simple constituents of $\usoc^q\left((V^*/V^*_{\aleph_0})_\lambda\right)$, see Proposition \ref{pr.form}.
%
\end{proof}

We have an analogous result regarding the tensorand $(V^*)_\mu\otimes V_\nu$ in \Cref{eq:inj}. Before stating it, we introduce the notation $\langle\bullet,\bullet\rangle$ for the bilinear form on $\sym$ making the basis $\{\lambda\}$ orthonormal.

\begin{lemma}\label{le.quot_filt'}
  The subquotient  $\usoc^q\left((V^*)_\mu\otimes V_\nu\right)$ is isomorphic to 
  \begin{equation*}
    \bigoplus\left((U_t)_{\mu_{(t+1)}^{\ell_{t+1}}} \otimes\cdots\otimes (U_0)_{\mu_{(1)}^{\ell_1}}\otimes V_{\mu_{(t+2)}^{\ell_{t+2}},\nu_{(1)}^{|\nu|-\tau}}\right)^{\oplus \langle\mu_{(t+3)},\nu_{(2)}\rangle},
  \end{equation*}
with summation over those choices of $\tau$ and $\ell_x$ such that
\begin{equation*}
\tau+\sum^{1}_{x=t+2}\ell_x=|\mu|
\end{equation*}
 and 
 \begin{equation}\label{eq.tausum}
 \tau+\sum^{1}_{x=t+1}x\ell_x=q-1.
\end{equation}
\end{lemma}
\begin{proof}
First, filtering $(V^*)_\mu$ alone using the same
  technique as in the proof of \Cref{le.quot_filt} we
  obtain a coarser filtration of $(V^*)_\mu\otimes V_\nu$ by
  \begin{equation*}
    0=F^0\subseteq F^1\subseteq F^2\subseteq \cdots
  \end{equation*}
  with the subquotient $F^q/F^{q-1}$ isomorphic to
  \begin{equation*}
    \bigoplus (U_t)_{\mu_{(t+1)}^{\ell_{t+1}}}\otimes\cdots\otimes (U_0)_{\mu_{(1)}^{\ell_1}}\otimes (V_*)_{\mu_{(t+2)}^{\ell_{t+2}}}\otimes V_\nu
  \end{equation*}
  for
  \begin{equation*}
    \sum^{1}_{x=t+2} \ell_x = |\mu| \text{ and } \sum^{1}_{x=t+1}x\ell_x=q-1.
  \end{equation*}
  Then we filter each subobject $(V_*)_\lambda\otimes V_\nu$ of
  one of the subquotients $F^q/F^{q-1}$ by its socle filtration.
  In the comultiplication-based notation and language, we use
  \cite[Theorem 2.3]{PS} says that
  $\usoc^p((V_*)_\lambda\otimes V_\nu)$ is isomorphic to
  \begin{equation*}
    \left(V_{\lambda_{(1)}^{|\mu|-p+1},\nu_{(1)}^{|\nu|-p+1}}\right)^{\oplus \langle \lambda_{(2)},\nu_{(2)}\rangle}. 
  \end{equation*}
  Splicing together these two filtration processes produces the
  claimed result.  Here (\ref{eq.tausum}) is responsible for identifying the submodule of $\usoc^q\left((V*)_\mu\otimes V_\nu\right)$.
%
\end{proof}

The following are more explicit versions of \Cref{le.quot_filt,le.quot_filt'}, involving the Littlewood-Richardson coefficients $N_{\mu,\nu}^\eta$.

\begin{qfbis}
 Let $u\leq t$, and let $\lambda$ and $\mu_x$, $0\le x\le t-u$, be partitions with
\begin{equation*}
\sum^{0}_{x=t-u}|\mu_x|=|\lambda|\,.
\end{equation*}
Set
\begin{equation*}
q:=1+\sum^{0}_{x=t-u}x|\mu_x|.  
\end{equation*}
Then the multiplicity of the simple object
\begin{equation*}
(U_t)_{\mu_{t-u}}\otimes\cdots\otimes (U_u)_{\mu_0}
\end{equation*}
 in the subquotient $\usoc^q((V^*/V^*_{\al_u})_\lambda)$  equals
  \begin{equation*}
    \sum N_{\mu_0,\mu_1}^{\alpha_1} N_{\alpha_1,\mu_2}^{\alpha_2}\cdots N_{\alpha_{t-u-1},\mu_{t-u}}^{\lambda}
  \end{equation*}
  with summation over repeated indices.


No other simples appear as constituents of $(V^*/V^*_{\al_u})_\lambda$. 
\end{qfbis}
\begin{proof}
  This is a reformulation of \Cref{le.quot_filt}, identifying $\ell_x$
  from that result to $|\mu_{x-1}|$ in the present one.

  Indeed, the multiplicity in question is the coefficient of
  \begin{equation*}
    \mu_{t-u}\otimes\cdots\otimes \mu_{0} 
  \end{equation*}
  in
  \begin{equation*}
    \Delta^{t-u}(\lambda)=\lambda_{(t-u+1)}\otimes\cdots\otimes \lambda_{(1)}. 
  \end{equation*}
  The very definition of the comultiplication in the ring of symmetric
  functions implies that this number is the multiplicity of
  $U_\lambda$ in the tensor product
  \begin{equation*}
    U_{\mu_{t-u}}\otimes\cdots\otimes U_{\mu_{0}}.
  \end{equation*}
  In turn, this is expressible in terms of Littlewood-Richardson
  coefficients as in the statement.
\end{proof}

\begin{qf'bis}
Let $\mu,\nu$ and $\eta_x$, $0\le x\le t$, $\xi$, $\zeta$ be partitions, and set
\begin{equation*}
 q:=1+(|\nu|-|\zeta|)+\sum_{x=0}^{t}(x+1)|\eta_j|.
\end{equation*}
Then the multiplicity of the simple object
$
V_{\eta_t,\cdots,\eta_0,\xi,\zeta}
$
in the subquotient  $\usoc^q((V^*)_\mu\otimes V_\nu)$ equals
  \begin{equation*}
   \sum N_{\eta_0,\eta_1}^{\pi_1}N_{\pi_1,\eta_2}^{\pi_2}\cdots N_{\pi_{t-1},\eta_t}^{\mu}N_{\xi,\delta}^{\pi_{t-1}}N_{\zeta,\delta}^{\nu},
  \end{equation*}
with summation over repeated indices. 

No other simples appear as constituents of $(V^*)_\mu\otimes V_\nu$.
\end{qf'bis}
\begin{proof}
  The deduction of this statement from that of \Cref{le.quot_filt'} is
  analogous to the previous proof: once more using the definition of
  the comultiplication, the multiplicity we are after is the sum
  \begin{equation}\label{eq:mults}
    \sum_\delta (\text{multiplicity of }\mu\text{ in }\eta_0\cdots\eta_t\xi\delta)\cdot (\text{multiplicity of }\nu\text{ in }\zeta\delta)
  \end{equation}
  where as before the Young diagram symbols stand for the
  corresponding Schur functions in $\sym$ and juxtaposition means
  multiplication therein.

  In turn, \Cref{eq:mults} is expressible as a sum of products of
  Littlewood-Richardson coefficients as claimed.
\end{proof}

Finally, the general result on the socle filtrations of the
indecomposable simple objects \Cref{eq:inj} is obtained by tensoring
together instances of \Cref{le.quot_filt,le.quot_filt'}.

\begin{proposition}\label{pr.quot_filt}
The subquotient  $\usoc^q(\widetilde{V}_{\lambda_t,\cdots,\lambda_0,\mu,\nu})$ of  \Cref{eq:inj} is isomorphic to
\begin{equation*}
  \sum_{u_x,y} \bigotimes_{x=t}^0\usoc^{u_x}((V^*/V^*_{\al_x})_{\lambda_x})\otimes \usoc^y((V^*)_{\mu}\otimes V_{\nu}),
\end{equation*}
with summation over all choices of $u_x$ and $y$ such that 
$
  \sum_{x=t}^0 (u_x-1)+(y-1)=q-1$.
\end{proposition}
\begin{proof}
  The ingredients are contained in \Cref{le.quot_filt,le.quot_filt'}
  and their proofs.
%
%
\end{proof}

\begin{remark}\label{re.non-sd}
  In \cite[$\S$3.7]{us} we show that when $\alpha=\aleph_0$, the
  category $\overline{\bT}_\alpha$ is not only Koszul, but is in fact
  {\it self-dual}: the quadratic algebra $E\overline{\bT}_{\aleph_0}$
  is anti-isomorphic to its quadratic dual.

  One consequence of the self-duality is that dimensions of $\ext$-groups can be read off from socle filtrations of indecomposable
  injectives: when $\alpha=\aleph_0$ the simples \Cref{eq:simple} are
  indexed by three Young diagrams, and \cite[Corollary 3.34]{us} then
  says that
  \begin{equation*}
    \dim \ext^q(V_{\lambda_0,\mu,\nu},V_{\lambda_0',\mu',\nu'}) 
  \end{equation*}
  equals the multiplicity of $V_{\lambda_0,\mu^{\perp},\nu}$ in the
  $\usoc^{q+1}$ subquotient of the injective hull of
  $V_{\lambda_0',(\mu')^\perp,\nu'}$; here $^\perp$ indicates passage to the transposed Young diagram.
  In other words, in order to compute the Ext-group $ \ext^q(V_{\lambda_0,\mu,\nu},V_{\lambda_0',\mu',\nu'})$, one simply passes to the socle filtration of $\widetilde{V}_{\lambda'_0,(\mu')^\perp,\nu'}$.  
  
  For $\alpha= \aleph_t$ for $t\geq 1$, no such ``transposing pattern'' computes $ \ext^q(V_{\lambda_t,\ldots,\lambda_0,\mu,\nu},V_{\lambda_t',\ldots,\lambda_0',\mu',\nu'})$.  Indeed, the subquotient $\usoc^3 V^*$ of the injective hull $V^*$
  of $V_*$ is nonzero, while the quotient
  $V^*/V_*$ is injective and hence all $\ext^2(\bullet,V_*)$ vanish.  The problem of computing the Ext-groups of pairs of simple objects in $\bT_{\aleph_t}$ is thus open for $t\geq 1$.
\end{remark}

We conclude this subsection by

\begin{example}\label{ex.examples}
Consider the objects $\widetilde{V}_{\emptyset,(1,1),\emptyset,\emptyset}$ and $\widetilde{V}_{\emptyset,\emptyset,(1,1),(1)}$ in $\bT_{\aleph_1}$.  They have respective socle filtrations
$$
\begin{tabular}{ccc}
&\multirow{4}{*}{$\phantom{some space}$}&$V_{(1,1),\emptyset,\emptyset,(1)}$\\[0.03in]\cline{3-3}
&&$V_{(1),(1),\emptyset,(1)}\oplus V_{(1),\emptyset,\emptyset,\emptyset}\vphantom{1^{1^1}}$\\[0.03in]\cline{3-3}
$V_{(1,1),\emptyset,\emptyset,\emptyset}$&&$V_{\emptyset,(1,1),\emptyset,(1)}\oplus V_{(1),\emptyset,(1),(1)} \oplus V_{\emptyset,(1),\emptyset,\emptyset}\vphantom{1^{1^1}}$\\[0.03in]\cline{3-3}\cline{1-1}
$V_{(1),(1),\emptyset,\emptyset}$&&$V_{\emptyset,(1),(1),(1)}\oplus V_{\emptyset,\emptyset,(1),\emptyset}\vphantom{1^{1^1}}$  \\[0.03in]\cline{3-3}\cline{1-1}
$V_{\emptyset,(1,1),\emptyset,\emptyset}$&&$V_{\emptyset,\emptyset,(1,1),(1)}\vphantom{1^{1^1}}$\\
\end{tabular}
$$
where $\usoc^s$ is displayed at level $s$ from the bottom (the bottom being the socle).  The object $$\tilde{V}_{\emptyset,(1,1),(1,1),(1)}=\tilde{V}_{\emptyset,(1,1),\emptyset,\emptyset}\otimes \tilde{V}_{\emptyset,\emptyset,(1,1),(1)}$$ has socle filtration
$$
\begin{array}{c}
	V_{(2,2),\emptyset,\emptyset,(1)}\oplus V_{(2,1,1),\emptyset,\emptyset,(1)}\oplus V_{(1,1,1,1),\emptyset,\emptyset,(1)}
	\\[0.03in]\hline
	2\,V_{(2,1),(1),\emptyset,(1)}\oplus 2\,V_{(1,1,1),(1),\emptyset,(1)}\oplus V_{(2,1),\emptyset,\emptyset,\emptyset}\oplus V_{(1,1,1),\emptyset,\emptyset,\emptyset}\vphantom{1^{1^1}}
	\\[0.03in]\hline\vphantom{1^{1^1}}
	3\,V_{(1,1),(1,1),\emptyset,(1)}\oplus V_{(2),(1,1),\emptyset,(1)} \oplus V_{(1,1),(2),\emptyset,(1)} \oplus V_{(2),(2),\emptyset,(1)}
	\\\vphantom{1^{1^1}}
	\oplus V_{(2),(1),\emptyset,\emptyset}\oplus 2\,V_{(1,1),(1),\emptyset,\emptyset}\oplus V_{(2,1),\emptyset,(1),(1)} \oplus V_{(1,1,1),\emptyset,(1),(1)}
	\\[0.03in]\hline\vphantom{1^{1^1}}
	V_{(2),(1),(1),(1)}\oplus 2\, V_{(1,1),(1),(1),(1)} \oplus V_{(1),(2),\emptyset,\emptyset} \oplus 2\, V_{(1),(1,1),\emptyset,\emptyset} 
	\\\vphantom{1^{1^1}}
	\oplus V_{(1,1),\emptyset,(1),\emptyset}\oplus 2\,V_{(1),(2,1),\emptyset,(1)} \oplus 2\,V_{(1),(1,1,1),\emptyset,(1)} 	
	\\[0.03in]\hline\vphantom{1^{1^1}}
	V_{\emptyset,(2,1),\emptyset,\emptyset}\oplus V_{\emptyset,(1,1,1),\emptyset,\emptyset}\oplus V_{(1),(2),(1),(1)} \oplus 2\,V_{(1),(1,1),(1),(1)}\oplus V_{(1),(1),(1),\emptyset}
	\\\vphantom{1^{1^1}}
	\oplus V_{\emptyset,(2,2),\emptyset,(1)} \oplus V_{\emptyset,(2,1,1),\emptyset,(1)}\oplus V_{\emptyset,(1,1,1,1),\emptyset,(1)}\oplus V_{(1,1),\emptyset,(1,1),(1)}
	\\[0.03in]\hline\vphantom{1^{1^1}}
	V_{\emptyset,(2,1),(1),(1)}\oplus V_{\emptyset,(1,1,1),(1),(1)}\oplus V_{\emptyset,(1,1),(1),\emptyset} \oplus V_{(1),(1),(1,1),(1)}
	\\[0.03in]\hline\vphantom{1^{1^1}}
	V_{\emptyset,(1,1),(1,1),(1)}
\end{array}
$$
\end{example}

\section{Universality}\label{se.univ}


Let $(\cD,\otimes,\mathbbm{1}_\cD)$ be a $\bK$-linear, abelian, symmetric
monoidal category with monoidal unit $\mathbbm{1}_\cD$, such that
all tensor functors $x\otimes \bullet$ are exact. 

\begin{convention}\label{cv.ccpl}
  A category $\cD$ as above is a {\it tensor category}, and a
  $\bK$-linear symmetric monoidal functor is a {\it tensor
    functor}. Similarly, a {\it tensor natural transformation} is one
  between tensor functors that respects all of the structure in the
  guessable fashion.
  
  A tensor category {\it has coproducts} if it has
  arbitrary direct sums preserved by all functors of the form
  $a\otimes \bullet$.
\end{convention}

With all of this in hand, the main result of the section reads as
follows.

\begin{theorem}\label{th.univ-gen}
  Let
  $\mathbf{q}:A\otimes B\to \mathbbm{1}_\cD$ be a morphism in $\cD$ and
  \begin{equation}\label{eq:filt-x-gen}
    0\subset A_0\subseteq A_1\subseteq  \cdots \subseteq A_a\subset A
  \end{equation}
  a transfinite filtration of $A$ indexed by ordinals from $0$ to $a$.  Set $\alpha=\aleph_a$.  Then
\begin{enumerate}
\renewcommand{\labelenumi}{(\alph{enumi})}
\item up to tensor natural isomorphism, there exists a unique left
  exact tensor functor $\mathcal{F}:\bT_{\alpha} \rightsquigarrow \cD$ turning the pairing
  $V^*\otimes V\to \bK$ into $\mathbf{q}$ and the transfinite socle filtration of $V^*$
  \begin{equation*}
    0\subset V^*_{\aleph_0}\subset \ldots \subset V^*_\alpha \subset V^*
  \end{equation*}
  into \Cref{eq:filt-x-gen};
\item if furthermore $\cD$ has coproducts in the sense of
  \Cref{cv.ccpl}, then the functor $\mathcal{F}$ extends uniquely to a
  coproduct-preserving tensor functor
  $\overline{\bT}_{\aleph_t}  \rightsquigarrow \cD$.
\end{enumerate}
\end{theorem}

%
%
%
%
%
%
%
%
%
%
%
%
%
%
%
%
%
%
%
%

We first prove Theorem~\ref{th.univ-gen} in the case $\alpha=t$ for $t\in\mathbb{Z}_{\geq 0}$.  The result is as follows.

\begin{theorem}\label{th.univ}
  Let $t$ be a nonnegative integer. Let $\mathbf{q}:A\otimes B\to \mathbbm{1}_\cD$ be a morphism in $\cD$ and 
  \begin{equation}\label{eq:filt-x}
    0\subset A_0\subseteq\cdots\subseteq A_{t}\subset A
  \end{equation}
a filtration in $\cD$ by subobjects of $A$. Then
\begin{enumerate}
\renewcommand{\labelenumi}{(\alph{enumi})}
\item up to tensor natural isomorphism, there exists a unique left
  exact tensor functor $\mathcal{F}:\bT_{\aleph_t}  \rightsquigarrow \cD$ turning the pairing
  $V^*\otimes V\to \bK$ into $\mathbf{q}$ and the socle filtration of $V^*$
  \begin{equation*}
    0\subset V^*_{\aleph_0}\subset\cdots\subset V^*_{\aleph_{t}}\subset V^*
  \end{equation*}
  into \Cref{eq:filt-x};
\item if furthermore $\cD$ has coproducts in the sense of
  \Cref{cv.ccpl}, then the functor $\mathcal{F}$ extends uniquely to a
  coproduct-preserving tensor functor
  $\overline{\bT}_{\aleph_t}  \rightsquigarrow \cD$.
\end{enumerate}
\end{theorem}

Before we embark on the proof of \Cref{th.univ}, it will
be convenient to slightly restate Corollary \ref{cor.bigone}(d) in
the spirit of \cite[$\S$3.4]{us}. For this purpose, we introduce
the tensor algebra
\begin{equation*}
  \mathrm{T}:=\mathrm{T}\left(\bigoplus_{s=t}^0(V^*/V_{\aleph_s}^*)\oplus V^*\oplus V\right) 
\end{equation*}
of the direct sum of the displayed ``degree-$1$'' indecomposable injectives. For $r\in\mathbb{Z}_{\geq 0}$, we
denote the degree-$r$ truncation of $\mathrm{T}$ by
$\mathrm{T}^{\le r}$.

Next, define 
\begin{equation*}
  \cA^r=\End_{\bT_{\aleph_t}}(\mathrm{T}^{\le r}),\ \cA=\bigcup_{r\in\mathbb{Z}_{\geq 0}} \cA^r,
\end{equation*}
where the inclusions $\cA^r\subset \cA^{r+1}$ are the obvious
ones (extension of an endomorphism by $0$). We have the following analogues of \cite[Definition 3.18]{us}
and \cite[Theorem 3.19]{us} (providing an alternate version of
Corollary \ref{cor.bigone}(d)).

\begin{definition}
  An $\cA$-module is {\it locally unitary} if it is unitary over some
  $\cA^r\subset \cA$.
\end{definition}

\begin{theorem}\label{th.kosz-fin-alg}
  The functor $\Hom_{\bT_{\aleph_t}}(\bullet,\mathrm{T})$ implements a contravariant
  equivalence between $\bT_{\aleph_t}$ and the category of
  finite-dimensional modules over $\cA$ which are locally
  unitary. \qedhere
\end{theorem}

\begin{remark}\label{re.coalg-choice}
  The inclusions $\cA^r\subset \cA$ split naturally, and hence give
  rise to inclusions $(\cA^r)^*\subset (\cA^{r+1})^*$ of dual
  finite-dimensional coalgebras. The union $\bigcup_r (\cA^r)^*$ can
  be chosen for our coalgebra $C$ from Corollary \ref{cor.bigone}(d).

  Moreover, the Koszulity of the coalgebra $C$ translates to the fact that all algebras 
  $\cA^r$ are Koszul and hence quadratic.
\end{remark}

We grade $\cA^r$ as follows:

\begin{definition}\label{def.grading}
  For $d\ge 0$, the degree-$d$ homogeneous component $\cA^r_d$ is the
  direct sum of all spaces of morphisms
\begin{equation*}
  \mathrm{T}^{\le r}\to Y\to Z\to \mathrm{T}^{\le r},
\end{equation*}
where 
\begin{itemize}
\item $Y$ and $Z$ are indecomposable direct summands of $\mathrm{T}^{\le r}$
  such that $\soc Z \in \cS_i$ and $\soc Y \in \cS_j$;
\item $d(i,j)=d$ (the defect from \Cref{def.def});
\item the outer arrows are the surjection and inclusion realizing respectively $Y$
  and $Z$ as direct summands.
\end{itemize}
The gradings of the various algebras $\cA^r$ are then compatible with
the inclusions $\cA^r\subset \cA^{r+1}$, so $\cA$ itself acquires an $\mathbb{Z}_{\geq0}$-grading.

We write $\deg(x)$ for the degree of an element $x\in \cA$.
\end{definition}

  The fact that \Cref{def.grading} does indeed define a grading
  follows from the triangle property of the defect, Corollary~\ref{cor.defecttriangle}(b).
  These gradings make the algebras $\cA^r$ Koszul, and the resulting
  grading on $\cA$ corresponds by duality to the grading on $C$ alluded to in
  Corollary \ref{cor.bigone}(d), see \Cref{re.coalg-choice}.

As part of our proof of \Cref{th.univ}, we will describe the
degree-one and degree-two components of $\cA$, as well as its
relations.  The latter are all quadratic by the Koszulity of the coalgebra $C$ from Corollary \ref{cor.bigone}(d).
 
\begin{notation}\label{not.indx}
  For $i=(n_s,n,m)$ we set 
  \begin{equation*}
    i_\ell = {\bf n}_i = (n_t,\cdots,n_0,n),\ i_r = m_i = m.
  \end{equation*}
We also set 
\begin{equation*}
  i_- = (n_t,\cdots,n_0,n-1,m-1),
\end{equation*}
\begin{equation*}
  i^s_{\pm}  = (n_t,\cdots, n_{s}+1,n_{s-1}-1,\cdots, n_0,n,m)
\end{equation*}
and 
\begin{equation*}
  i^s_+  = (n_t,\cdots, n_{s}+1,n_{s-1},\cdots, n_{0},n,m)
\end{equation*}
for $0\le s\le t$, where for the purpose of defining $i^0_{\pm}$ we
regard the component $n$ as $n_{-1}$.
\end{notation}

\begin{notation}\label{not.sym}
Let $S_p$ be the symmetric group on $p$ symbols. We denote multiple products $S_p\times S_u\times S_q\times\cdots$ by
$S_{p,u,q,\cdots}$.  For $i=\left(n_t,\ldots,n_0,n,m\right)\in I$, we set $S_{i}=S_{n_t,\ldots,n_0,n,m}$.  In addition, in the rest of the paper, $\Hom(\bullet,\bullet)=\Hom_{\bT_{\aleph_t}}(\bullet,\bullet)$ and $\End(\bullet)=\End_{\bT_{\aleph_t}}(\bullet)$.
\end{notation}

As an immediate consequence of \Cref{pr.simples} and the
classification of simple objects and their injective envelopes given
in \Cref{th.indec_inj-gen}, we have the following analogue of
\cite[Lemma 3.22]{us}, describing the degree-zero component
of $\cA$.

\begin{lemma}\label{le.deg-0}
  Let $i=\left(n_t,\ldots,n_0,n,m\right)\in I$. The endomorphism algebra of the injective
  object 
  $$X_i=\bigoplus_{s=t}^0 \,\left(V^*/V^*_{\aleph_s}\right)^{\otimes n_s} \otimes \left(V^*\right)^{\otimes n}\otimes V^{\otimes m}$$ is isomorphic to the group algebra $\bK S_i$.  \qedhere
\end{lemma}

 According to the
description of defect-one pairs $j\prec i$ given in the proof of \Cref{pr.form},
the morphisms that make up the degree-one component $\cA_1$ are
qualitatively of two types (matching the two bullet points in the proof of
\Cref{pr.form}):

\begin{itemize}
  \item morphisms $X_i\to X_{i_-}$;
  \item morphisms $X_i\to X_{i^s_{\pm}}$ for some $1\le s\le t$.
\end{itemize}

Corresponding examples of such morphisms are:
\begin{itemize}
\item the evaluation morphism $\phi_{p,q}:X_i\to X_{i_-}$ of the
  $p^{th}$ tensorand  $V^*$  against the $q^{th}$ tensorand $V$ for
  some choice of $1\le p\le n$ and $1\le q\le m$;
\item the surjection morphism $\pi^s_{p,q}:X_i\to X_{i^s_{\pm}}$ of the
  $p^{th}$  tensorand $V^*/V^*_{\al_{s-1}}$ onto the $q^{th}$ tensorand
  $V^*/V^*_{\al_{s}}$ for a choice of
  \begin{equation*}
    0\le s\le t,\quad 1\le p\le n_{s-1},\quad 0\le q\le n_{s}. 
  \end{equation*}
\end{itemize}
More precisely, $\pi^s_{p,q}$ first implements a surjection
$V^*/V^*_{\al_{s-1}}\to V^*/V^*_{\al_{s}}$ defined on the $p^{th}$
tensorand $V^*/V^*_{\al_{s-1}}$ in its domain, and then inserts the
result of that surjection as the $q^{th}$ tensorand of type
$V^*/V^*_{\al_{s}}$ in the codomain, without altering the order of the
other tensorands.

\begin{remark}\label{re.full}
  It follows from \cite[Lemma 2.19]{us} that the algebra $\mathcal{A}$ is generated by all evaluation morphisms $\phi_{p,q}$, all surjection morphisms $\pi^s_{p,q}$, and all permutations of tensorands of the objects $X_i$.
\end{remark}

In the following discussion we will often choose $i,s,p,q$ as above. In such a setting, the group $S_{i_-}$ will be
understood to be embedded into $S_i$ as the group of permutations
fixing
\begin{equation*}
  p\in \{1,\cdots,n\} \quad\text{and} \quad q\in \{1,\cdots,m\}. 
\end{equation*}
Similarly, we embed $S_i$ and $S_{i^s_{\pm}}$ into $S_{i^s_+}$ in the
obvious fashion.

\begin{lemma}\label{le.deg-1}
  Let $i=\left(n_t,\ldots,n_0,n,m\right)\in I$ and $s,p,q$ be as above.  Then, we have the following description for Hom-spaces of degree-one elements in $\cA$.
  \begin{enumerate}
    \renewcommand{\labelenumi}{(\alph{enumi})}
  \item The identification $\phi_{p,q}\mapsto 1$ extends to a bimodule
    isomorphism $\Hom(X_i,X_{i_-})\cong \bK S_i$, where both sides are
    equipped with standard bimodule structures over
    \begin{equation*}
      \End(X_{i_-})\cong \bK S_{i_-} \quad\text{and}\quad \End(X_i)\cong \bK S_i. 
    \end{equation*}
  \item The identification $\pi^s_{p,q}\mapsto 1$ extends to a
    bimodule isomorphism
    \begin{equation*}
      \Hom(X_i,X_{i^s_{\pm}})\cong \bK S_{i^s_+}\cong \mathrm{Ind}_{S_{n_s}}^{S_{n_s+1}}\bK S_i,
    \end{equation*}
    where both sides are equipped with standard bimodule structures
    over
    \begin{equation*}
      \End(X_{i^s_{\pm}})\cong \bK S_{i^s_{\pm}} \quad\text{and}\quad \End(X_i)\cong \bK S_i. 
    \end{equation*}
  \end{enumerate}
\end{lemma}
\begin{proof}
  The identifications in question give rise to morphisms of bimodules
  \begin{equation*}
    \bK S_i\to \Hom(X_i,X_{i_-})\text{ and } \bK S_{i^s_{++}}\to \Hom(X_i,X_{i^s_{\pm}})
  \end{equation*}
  which are surjective by \Cref{re.full}. The proof that the maps are also
  injective proceeds virtually identically to the corresponding
  argument in the proof of \cite[Lemma 3.24]{us}.
\end{proof}

The following result provides a piecewise description of the
degree-two component of the tensor algebra of $\cA_1$ over $\cA_0$. It
parallels and generalizes \cite[Lemma 3.25]{us}, and it follows
routinely from the identifications made in
\Cref{le.deg-1}.

Since we have to apply the operations $i\mapsto i_-$,
$i\mapsto i^s_{\pm}$ and $i^s_+$ repeatedly, we will simply
concatenate superscripts and $\pm$ subscripts. Note that with this
notation, we have $i^s_{-\pm}=i^s_{\pm -}$.

\begin{lemma}\label{le.deg-2}
  Let $i\in I$. The tensor products of the
  spaces described in \Cref{le.deg-1} are as follows.
  \begin{enumerate}
        \renewcommand{\labelenumi}{(\alph{enumi})}
      \item The space
        \begin{equation}
          \label{eq:1}
          \Hom(X_{i_-},X_{i_{--}})\otimes_{\End(X_{i_-})}\Hom(X_i,X_{i_-})
        \end{equation}
        is isomorphic to $\bK S_i$ as an $(S_{i_{--}},S_i)$-bimodule.
      \item The spaces
        \begin{equation}
          \label{eq:2}
          \Hom(X_{i_-},X_{i^s_{-\pm}})\otimes_{\End(X_{i_-})}\Hom(X_i,X_{i_-})
        \end{equation}
        and
        \begin{equation}
          \label{eq:3}
          \Hom(X_{i^s_{\pm}},X_{i^s_{\pm-}})\otimes_{\End(X_{i^s_{\pm}})}\Hom(X_i,X_{i^s_{\pm}})
        \end{equation}
        are both isomorphic to
        \begin{equation*}
          \bK S_{i^s_+} \cong \mathrm{Ind}_{S_{n_s}}^{S_{n_s+1}}\bK S_i
        \end{equation*}
        as $(S_{i^s_{\pm-}},S_i)$-bimodules.
      \item Now let $0\le r, s\le t$ and set $j=i^{rs}_{\pm\pm}$. Then the space
        \begin{equation}
          \label{eq:4}
          \Hom(X_{i^r_{\pm}},X_j)\otimes_{\End(X_{i^r_{\pm}})}\Hom(X_i,X_{i^r_{\pm}})
        \end{equation}
        is isomorphic to $\bK S^{rs}_{++}$ as an $(S^{rs}_{++},S_i)$-bimodule.
        \qedhere
  \end{enumerate}
\end{lemma}

Next, we describe the subspaces of the tensor products from
\Cref{le.deg-2} that are annihilated upon composing morphisms in
$\cA$. We have three types of relations, corresponding to parts (a) to
(c) of \Cref{le.deg-2}.


\begin{convention}\label{cv.ppi}
  We drop the $p$ and $q$ subscripts from the morphisms $\phi_{p,q}$
  and $\pi^s_{p,q}$ when the latter are understood to involve only the
  rightmost relevant tensorands. So for instance,
  $\phi:X_i\to X_{i_-}$ is the evaluation of the $n^{th}$
  tensorand $V^*$ against the $m^{th}$ tensorand $V$ in $X_i$.

  The generators of the spaces in parts (a) and (b) of \Cref{le.deg-1}
  are always assumed (unless specified otherwise) to be $\phi$ and
  $\pi$ respectively.
\end{convention}

\begin{notation}\label{not.theta}
  For $i\in I$ and $j=i^{rr'}_{\pm\pm}$ for some
  $r,r'\in \{0,\cdots,t\}$, we denote by $\Theta_{i,j}$ the set of
  those $s$ for which
  \begin{equation*}
    d(j,i^s_\pm) = d(i^s_\pm,i) = 1. 
  \end{equation*}
\end{notation}

In other words, $\Theta_{i,j}$ is the set of all $s$ for which
$X_{i^s_{\pm}}$ can appear as an intermediate object for morphisms
$X_i\to X_j=X_{i^{rr'}_{\pm\pm}}$ obtained by composition from tensor
products as in part (c) of \Cref{le.deg-2}.

\begin{lemma}\label{le.rel}
  Let $i=\left(n_t,\ldots,n_0,n,m\right)\in I$. The degree-two relations of
  $\cA$ can be described as follows.
  \begin{enumerate}
    \renewcommand{\labelenumi}{(\alph{enumi})}
  \item The map
    \begin{equation*}
      \Cref{eq:1}\to \Hom(X_i,X_{i_{--}})
    \end{equation*}
    is surjective, with kernel generated by
    \begin{equation*}
      \phi\otimes\phi -(\phi\otimes\phi)\circ(n,n-1)(m,m-1)
    \end{equation*}
    as an $(S_{i_{--}},S_i)$-bimodule, where $(n,n-1)$ and
    $(m,m-1)$ are transpositions in the factors $S_{n}$ and $S_m$
    respectively of
    \begin{equation*}
      S_i = S_{n_t}\times \cdots\times S_{n_0}\times S_{n}\times S_m. 
    \end{equation*}
  \item The map
    \begin{equation*}
      \Cref{eq:2}\oplus \Cref{eq:3}\to \Hom(X_i,X_{i^s_{-\pm}})
    \end{equation*}
    is surjective, with kernel generated by
    \begin{equation*}
      \pi^s\otimes\phi - (\phi\otimes\pi^s)\circ \tau
    \end{equation*}
    as an $(S_{i^s_{\pm-}},S_i)$-bimodule, where
    \begin{equation*}
      \tau=(n,n-1)\in S_{n}\subset S_i
    \end{equation*}
    if $s=0$ and $\tau=\mathrm{id}$ otherwise.
  \item Let $j=i^{ss'}_{\pm\pm}$ for some $0\le s\le s'\le t$. The map
    \begin{equation}\label{eq:comp}
      \bigoplus_{r\in \Theta_{i,j}}\Cref{eq:4}\to \Hom(X_i,X_j)
    \end{equation}
    is a surjective morphism of $(S_j,S_i)$-bimodules, and we have the
    following cases:
    \begin{enumerate}
      \renewcommand{\labelenumii}{(\arabic{enumii})}
    \item if $s'-s\ge 2$ then $\Theta_{i,j}=\{s,s'\}$, and the
      kernel of the bimodule map \Cref{eq:comp} is generated by
      \begin{equation*}
        \pi^s\otimes\pi^{s'}-\pi^{s'}\otimes \pi^s; 
      \end{equation*}
    \item if $s'=s+1$ then $\Theta_{i,j}=\{s,s+1\}$ again, and the
      kernel of \Cref{eq:comp} is generated by
      \begin{equation}\label{eq:ss1}
        \pi^s\otimes\pi^{s+1}-\sigma\circ (\pi^{s+1}\otimes\pi^s), 
      \end{equation}
      where
      \begin{equation*}
        \sigma=(n_s+1,n_s)\in S_{n_s+1}\subset S_{i^{ss'}_{++}}=S_{n_t}\times\cdots\times S_{n_{s+1}+1}\times S_{n_s+1}\times\cdots \times S_{n_0}\times S_n\times S_m;
      \end{equation*}      
    \item if $s=s'$ then $\Theta_{i,j}=\{s\}$, and the kernel of
      \Cref{eq:comp} is generated by
      \begin{equation*}
        \sigma\circ(\pi^s\otimes\pi^s)-(\pi^s\otimes \pi^s)\circ\tau,
      \end{equation*}
      where
      \begin{equation*}
        \tau=(n_{s-1},n_{s-1}-1)\in S_{n_{s-1}}\subset S_i = S_{n_t}\times\cdots \times S_{n_0}\times S_{n}\times S_m
      \end{equation*}
      and
      \begin{equation*}
        \sigma=(n_s+2,n_s+1)\in S_{n_s+2}\subset S_{i^{ss}_{++}} = S_{n_t}\times\cdots\times S_{n_s+2}\times\cdots\times S_{n_0}\times S_{n}\times S_m.
      \end{equation*}      
    \end{enumerate}
  \end{enumerate}
\end{lemma}
\begin{proof}
  Surjectivity follows in all cases from the fact that $\cA$ is
  generated in degree one; in turn, this is a consequence of
  \Cref{th.kosz-gen}.

  As for the description of the kernels, we will prove one of the
  several cases listed in the statement, the rest being entirely
  analogous.
  Note that the types of compositions covered by parts (a) and (b) of
  the lemma are qualitatively similar to those in \cite[Lemma 3.26,
  (a) and (b)]{us} ((a) refers to compositions of evaluations while
  (b) refers to composing evaluations and surjections with source and
  target of the form $V^*/V^*_{\bullet}$). For this reason, we focus
  on part (c). Point (3) of the latter is analogous to \cite[Lemma
  3.26 (c)]{us}, so the new phenomena occur in parts (1) and (2) of
  (c). Finally, we will tackle part (2) as representative of the
  essence of the argument.

  Since the kernel of \Cref{eq:comp} is easily seen to contain the
  $(S_j,S_i)$-bimodule generated by \Cref{eq:ss1}, it suffices to
  prove the opposite containment. As in \cite[Lemma 3.26]{us} we do
  this by a dimension count, i.e. by proving a lower
  bound on the dimension of the image $\Hom(X_i,X_j)$ of
  \Cref{eq:comp}.
  
  The two summands of \Cref{eq:ss1} belong respectively to the two direct
  summands of
  \begin{equation*}
    \left(\Hom(X_{i^{s+1}_{\pm}},X_j)\otimes
      \Hom(X_i,X_{i^{s+1}_{\pm}})\right)\oplus \left(\Hom(X_{i^s_{\pm}},X_j)\otimes\Hom(X_i,X_{i^s_{\pm}})\right), 
  \end{equation*}
  which upon making the identification from point (c) of
  \Cref{le.deg-2} is isomorphic to
  \begin{equation*}
    \left(\bK S_{i^{ss'}_{++}}\right)^{\oplus 2} = \bK [S_{n_t}\times\cdots \times S_{n_{s+1}+1}\times S_{n_s+1}\times\cdots\times S_{n_0}\times S_n\times S_m]^{\oplus 2}. 
  \end{equation*}
  Under this identification, \Cref{eq:ss1} is simply
  $1\oplus (-\sigma)$, with $\sigma=(n_s+1,n_s)$ in the $S_{n_s+1}$
  factor of the second $\bK S_{i^{ss'}_{++}}$ summand.

  The $(S_j,S_i)$-bimodule generated by $1\oplus(-\sigma)$ coincides
  with the left $S_{i^{ss'}_{++}}$-module generated by it, and hence
  has dimension
  \begin{equation*}
    \mathrm{dim} \left(\bK S_{i^{ss'}_{++}}\right) = n_t!\cdots (n_{s+1}+1)! (n_s+1)!\cdots n_0!n!m!
  \end{equation*}
  It thus suffices to show that the dimension of the image
  $\Hom(X_i,X_j)$ of \Cref{eq:comp} is at least
  \begin{equation*}
    \mathrm{dim} \left(\bK S_{i^{ss'}_{++}}\right)^{\oplus 2}-\mathrm{dim} \left(\bK S_{i^{ss'}_{++}}\right) =
    \mathrm{dim} \left(\bK S_{i^{ss'}_{++}}\right).
  \end{equation*}
  Since the tuples $i$ and $j$ only differ in their index-$(s+1)$ and
  index-$(s-1)$-entries, we can argue as in \Cref{le.deg-0,le.deg-1}
  that we have
  \begin{equation*}
    \Hom(X_i,X_j)\cong \bK[S_{n_t}\cdots S_{n_{s+2}}\times S_{n_{s-2}}\cdots S_{n_0}\times S_{n}\times S_m]\otimes \Hom(X_{i'},X_{j'}),
  \end{equation*}
  where $i'$ and $j'$ are the tuples obtained from $i$ and $j$
  respectively by substituting zeros for $m$ and $n_r$,
  $r\not\in\{s,s\pm 1\}$. We can thus assume that the only possibly
  nonzero entries in $i$ and $j$ are those with indices $s-1$, $s$
  and $s+1$; in other words, we suppose that
  \begin{equation*}
    i=(0,\cdots,n_{s+1},n_s,n_{s-1},\cdots,0),\quad j=(0,\cdots,n_{s+1}+1,n_s,n_{s-1}-1,\cdots,0).
  \end{equation*}
  In this setting, we then have to prove
  \begin{equation}\label{eq:dim-ct}
    \Hom(X_i,X_j)\ge (n_{s+1}+1)! (n_{s}+1)! (n_{s-1})!
  \end{equation}
  Now consider the $(n_s+1)!$ elements of $\Hom(X_i,X_j)$ obtained as
  follows:
  \begin{itemize}
  \item surject the rightmost tensorand $V^*/V^*_{\aleph_{s-1}}$ onto the rightmost tensorand $V^*/V^*_{\aleph_s}$;
  \item apply a permutation in $S_{n_s+1}$ to the resulting tensorand $(V^*/V^*_{\aleph_s})^{\otimes (n_s+1)}$;
  \item surject the rightmost tensorand $V^*/V^*_{\aleph_{s}}$ onto the
    rightmost tensorand $V^*/V^*_{\aleph_{s+1}}$.
  \end{itemize}
  By choosing a basis for $V^*$ compatible with the filtration by the
  $V^*_{\aleph_r}$ and examining the action of $\Hom(X_i,X_j)$ on the
  resulting tensor product basis, it is easy to see that the
  $(n_s+1)!$ morphisms we have just described span a free
  $(S_{n_{s+1}+1},S_{n_{s-1}})$-sub-bimodule of $\Hom(X_i,X_j)$. This
  proves the desired dimension count \Cref{eq:dim-ct}.
\end{proof}

\begin{proof-of-univ}
  We follow the same strategy as in the proof of \cite[Theorem
  3.22]{us}.

  First, note that the conclusion would follow from an application of
  \cite[Theorem 2.22]{us} provided we can show that we can define a
  symmetric monoidal functor (also denoted by $\mathcal{F}$ by a slight
  notational abuse) to $\cD$ from the tensor category with objects
  $X_i$ and Hom-spaces generated by the morphisms from \Cref{le.deg-1}.

  Note that such a tensor functor is constrained up to tensor natural
  isomorphism by the requirements of the statement:
  \begin{equation*}
    \mathcal{F}(V)=B,\quad \mathcal{F}\left(V^*_s\right) = A_s,\quad \mathcal{F}(V^*\otimes V\to \bK) = \mathbf{q}, 
  \end{equation*}
  as well as
  \begin{equation*}
    \mathcal{F}(V^*/V^*_{\aleph_{s-1}}\to V^*/V^*_{\aleph_{s}}) = A/A_{s-1}\to A/A_{s}\,.
  \end{equation*}

  The only thing to check is that such a definition is consistent,
  i.e. that the relations satisfied in $\bT_{\aleph_t}$ by the
  morphisms in \Cref{le.deg-1} are satisfied by the corresponding
  morphisms in $\cD$.

  The Koszulity result proven in \Cref{th.kosz-gen}, reinterpreted via
  Corollary \ref{cor.bigone}(d) and \Cref{re.coalg-choice}, shows that it suffices to prove
  that the quadratic relations described by \Cref{le.rel} are
  satisfied in $\mathcal{D}$.  But this follows directly from the definition of $\mathcal{D}$ as a tensor category.
\end{proof-of-univ}

The next lemma is a small amplification of \Cref{th.univ}
providing necessary and sufficient conditions for the functors $\mathcal{F}$
therein to be exact (rather than only left exact).

\begin{lemma}\label{le.univ}
  The functor $\mathcal{F}:\bT_{\aleph_t}\rightsquigarrow \cD$ from \Cref{th.univ} is exact
  if and only if its restriction to the full subcategory on the
  indecomposable injective objects \Cref{eq:inj} of $\bT_{\aleph_t}$
  is exact.
\end{lemma}
\begin{proof}
  One implication (exactness of $\mathcal{F}$ $\Rightarrow$ exactness of its
  restriction) is immediate, whereas the other one follows from the
  description of the right derived functors $R^*\mathcal{F}$ via injective
  resolutions in $\bT_{\aleph_t}$: the hypothesis ensures that the
  injective resolution
  \begin{equation*}
    x\to J^0\to J^1\to\cdots
  \end{equation*}
  of any object in $\bT_{\aleph_t}$ is turned by $\mathcal{F}$ into an exact
  sequence, hence the conclusion that the higher derived functors
  $R^p\mathcal{F}(x)=H^p\big(\mathcal{F}(J^\bullet)\big)$ vanish.
\end{proof}


Finally, we are ready for the proof of \Cref{th.univ-gen}.

\begin{proof-of-univ-gen}
  For each tuple ${\bf s}$
  \begin{equation*}
    \beta_0<\cdots<\beta_t\le \alpha
  \end{equation*}
  of infinite cardinal numbers \Cref{th.univ} provides a functor
  $\bT_{\aleph_t}\rightsquigarrow \bT_{\alpha}$ that easily seen to be
  a full exact embedding (exactness follows from \Cref{le.univ}).

  Applying \Cref{th.univ} to all such finite-length subcategories
  $\bT_{\bf s}\cong \bT_{\aleph_t}\subseteq \bT_\alpha$ we obtain
  functors $\mathcal{F}_{\bf s}:\bT_{\bf s}\rightsquigarrow \cD$
  turning $V^*$ into $A$, the filtration
  \begin{equation*}
    V^*_{\beta_0}\subset \cdots\subset V^*_{\beta_t}
  \end{equation*}
into 
\begin{equation*}
  A_{\beta_0}\subseteq \cdots\subseteq A_{\beta_t}
\end{equation*}
and the pairing $V^*\otimes V\to \bK$ into $\mathbf{q}$.

We regard tuples ${\bf s}$ as forming a directed system by
inclusion. The assignment ${\bf s}\to \bT_{\bf s}$ is increasing in
the sense that if ${\bf s'}$ contains ${\bf s}$ then
$\bT_{\bf s}\subset \bT_{\bf s'}$, and $\bT_{\alpha}$ can be expressed
as the direct limit
\begin{equation}\label{eq:t-unn}
  \bT_{\alpha}=\lim_{\substack{{\longrightarrow}\\{s}}} \bT_{\bf s}. 
\end{equation}

The uniqueness statements in \Cref{th.univ} ensure that the functors
$\mathcal{F}_{\bf s}$ are compatible, in the sense that
$\mathcal{F}_{\bf s}$ and $\mathcal{F}_{\bf s'}$ are (up to tensor
natural isomorphism) restrictions of $\mathcal{F}_{\bf s''}$ for any
tuple ${\bf s''}$ that contains ${\bf s}$ and ${\bf s'}$.

The conclusion follows from \Cref{eq:t-unn} by taking $\mathcal{F}$ to
be the unique (up to tensor natural isomorphism) tensor functor
restricting to $\mathcal{F}_{\bf s}$ on each $\bT_{\bf s}$.
\end{proof-of-univ-gen}


\bibliographystyle{abbrv}
\addcontentsline{toc}{section}{References}

\def\cftil#1{\ifmmode\setbox7\hbox{$\accent"5E#1$}\else
  \setbox7\hbox{\accent"5E#1}\penalty 10000\relax\fi\raise 1\ht7
  \hbox{\lower1.15ex\hbox to 1\wd7{\hss\accent"7E\hss}}\penalty 10000
  \hskip-1\wd7\penalty 10000\box7}

\Addresses

\end{document}